\documentclass{amsart}
\usepackage[utf8]{inputenc}
\usepackage[T1]{fontenc}
\usepackage{geometry}
\usepackage{amssymb}
\usepackage{amsmath}
\usepackage{amsfonts}
\usepackage{latexsym}
\usepackage{amsthm}
\usepackage{mathrsfs} 
\usepackage[colorlinks, urlcolor=blue, citecolor=black, linkcolor=black, breaklinks]{hyperref} 

\usepackage{algorithm}
\usepackage[noend]{algpseudocode}
\usepackage{enumerate}

\algnewcommand{\LineComment}[1]{\State \(\triangleright\) #1}

\newcommand{\BigO}[1]{\ensuremath{\operatorname{O}\bigl(#1\bigr)}}

\newtheorem{theorem}{Theorem}[section]
\newtheorem{lemma}[theorem]{Lemma}
\newtheorem{proposition}[theorem]{Proposition}

\theoremstyle{definition}

\newtheorem{definition}[theorem]{Definition}
\newtheorem{remark}[theorem]{Remark}

\def\F{\mathbb{F}}
\def\P{\mathbb{P}}

\newcommand{\Ima}{\textrm{Im}}

\newcommand{\udots}{\mathinner{\mskip1mu\raise1pt\vbox{\kern7pt\hbox{.}}
  \mskip2mu\raise4pt\hbox{.}\mskip2mu\raise7pt\hbox{.}\mskip1mu}}
  
\usepackage{xcolor}

\begin{document}  

\title[On The Effective Construction of Asymmetric Chudnovsky Multiplication Algorithms in Finite Fields Without Derivated Evaluation]{ On The Effective Construction of Asymmetric Chudnovsky Multiplication Algorithms in  Finite Fields Without Derivated Evaluation}

\author{St\'ephane Ballet} 
\address{Aix Marseille Univ, CNRS, Centrale Marseille, I2M, Marseille,  France} 
\email{stephane.ballet@univ-amu.fr}
\author{Nicolas Baudru}
\address{Aix Marseille Univ, CNRS, Centrale Marseille, LIF, Marseille, France}
\email{nicolas.baudru@univ-amu.fr}
\author{Alexis Bonnecaze}
\address{Aix Marseille Univ, CNRS, Centrale Marseille, I2M, Marseille,  France} 
\email{alexis.bonnecaze@univ-amu.fr}
\author{Mila Tukumuli}
\address{  } 
\email{tukumulimila@gmail.com}
\maketitle

\begin{abstract}
\noindent

The Chudnovsky and Chudnovsky algorithm for the multiplication in extensions of finite fields 
provides a bilinear complexity which is uniformly linear whith respect to the degree of the extension.
Recently, Randriambololona has generalized the method, allowing asymmetry in the interpolation procedure
and leading to new upper bounds on the bilinear complexity.
We describe the effective algorithm of this asymmetric method,  without derivated evaluation.
Finally, we give examples with the finite field $\F_{16^{13}}$ using only rational places, $\F_{4^{13}}$ 
using also places of degree two and $\F_{2^{13}}$ using also places of degree four.

\end{abstract}

Keywords: Multiplication algorithm, bilinear complexity,  interpolation on algebraic curve, finite field.

\section{Introduction}\label{intro}

Let $q$ be a prime power, ${\mathbb F}_q$ the finite field with $q$ elements and ${\mathbb F}_{q^n}$ the degree $n$ extension of ${\mathbb F}_q$.
Among all algorithms of multiplications in ${\mathbb F}_{q^n}$, those based on Chudnovsky-Chudnovsky \cite{Chudnovsky} 
method are known to provide the lowest bilinear complexity. This method is based on interpolation on algebraic curves defined over a finite field  and provides a 
bilinear complexity which is linear in $n$. The original algorithm uses only points of degree
1, with multiplicity 1. Ballet and Rolland~\cite{Ballet4, Ballet5} and Arnaud~\cite{Arnaud} improved the algorithm introducing interpolation 
at points of higher degree or higher multiplicity. The symmetry of the construction involves 2-torsion points that represent an obstacle 
to the improvement of upper bilinear complexity bounds. To eliminate this difficulty,  Randriambololona~\cite{Randriam} allowed asymmetry in the interpolation procedure, 
and then Pieltant and Randriambololona~\cite{pira} derived new bounds, uniform in $q$, of the bilinear complexity. Unlike symmetric constructions, no effective implementation 
of this asymmetric  construction has been done yet. In this article, we construct explicitly such multiplication algorithms for finite extensions of finite fields.

\subsection{Multiplication algorithm and tensor rank}

Let $q$ be a prime power, ${\mathbb F}_q$ the finite field with $q$ elements and ${\mathbb F}_{q^n}$ the degree $n$ extension of ${\mathbb F}_q$.
The multiplication of two elements of $\F_{q^n}$ is 
an $\F_q$-bilinear application from $\F_{q^n} \times \F_{q^n}$ onto $\F_{q^n}$.
Then it can be considered as an $\F_q$-linear application from the tensor product 
${\F_{q^n} \otimes_{\F_q} \F_{q^n}}$
onto $\F_{q^n}$. Consequently, it can also be  considered as an element 
$T_m$ of ${{\F_{q^n}}^\star \otimes_{\F_q} {\F_{q^n}}^\star \otimes_{\F_q} \F_{q^n}}$ where $\star$ denotes the dual.
When $T_m$ is written
\begin{equation}\label{tensor}
T_m=\sum_{i=1}^{r} x_i^\star\otimes y_i^\star\otimes c_i,
\end{equation}
where the $r$ elements $x_i^\star$ as well as the $r$ elements $y_i^\star$
are in the dual ${\F_{q^n}}^\star$ of $\F_{q^n}$ while the $r$ elements $c_i$ are in $\F_{q^n}$,
the following holds for any ${x,y \in \F_{q^n}}$:
$$
x\cdot y=\sum_{i=1}^r x_i^\star(x) y_i^\star(y) c_i.
$$
The decomposition (\ref{tensor}) is not unique. 

\begin{definition}\label{defalgobil}
Every expression
$$
x\cdot y=\sum_{i=1}^r x_i^\star(x) y_i^\star(y) c_i
$$
defines a bilinear multiplication algorithm ${\mathcal U}$ of bilinear complexity $\mu({\mathcal U})=r$.
Such an algorithm is said symmetric if $x_i=y_i$ for all $i$.

\end{definition}

\begin{definition}\label{defcomplexbil}
The minimal number of summands in a decomposition of the tensor $T_m$ of the multiplication
is called the bilinear complexity (resp. symmetric bilinear complexity) of the multiplication and is denoted by
$\mu_{q}(n)$ (resp. $\mu^{sym}_{q}(n)$):
$$
\mu_{q}(n)= \min_{{\mathcal U}} \mu({\mathcal U})
$$ where ${\mathcal U}$ is running over all bilinear multiplication algorithms (resp. all bilinear symmmetric multiplication algorithms) in $\F_{q^n}$ over $\F_q$.
\end{definition}

\subsection{Known results}
In 1979, Winograd proved~\cite{Winograd} that optimal multiplication algorithms realizing the lowest bilinear bound belong to the class of interpolation algorithms. Then, in 1988, 
 Chudnovsky and Chudnovsky introduced a method~\cite{Chudnovsky}  to prove the linearity~\cite{Ballet2} of the bilinear complexity of the multiplication in finite extensions of a finite field. 
 In doing so, they proposed the first known  multiplication algorithm using interpolation to algebraic function fields (of one variable) over $ \F_{q} $. 
 This  original algorithm only uses points of degree 1, with multiplicity 1. Later, several studies focused on the qualitative improvement of this algorithm 
 (for example~\cite{Ballet4, Arnaud, CenkOzbudak}) 
  allowing interpolation at points of higher degree, or with higher multiplicity. In parallel,  improvements of upper bounds (for example~\cite{Ballet5, BalletPieltant, bapirasi}) 
  and asymptotic upper bounds (for example~\cite{shtsvl, Ballet8}) of the bilinear complexity were obtained. 
 
The first known effective finite field multiplication through interpolation on algebraic curves was proposed by Shokrollahi and Baum ~\cite{ShokBaum}. They used the Fermat curve $ x^3 + y^3 = 1$ 
to construct a multiplication algorithm over $\F_{4^4}$ with 8 bilinear multiplications. In~\cite{Ballet3}, Ballet proposed one over $ \F_{16^n} $ where $ n~\in~[13,14,15]$, 
using the hyperelliptic curve  $y^2 + y = x^5$ of genus 2, with $2n + 1$  bilinear multiplications. Notice that these aforementioned two algorithms only used rational points, with multiplicity 1. 
In 2009, Cenk and \"{O}zbudak proposed in~\cite{CenkOzbudak} an explicit elliptic multiplication algorithm in $ \F_{3^9} $ with 26 bilinear multiplications. 
To this end, they used the elliptic curve $ y^2 = x^3 +  x  + 2 $, combining the ideas of using  points of higher degree and higher multiplicity. 
In fact, few studies have been devoted to the effective construction of Chudnovsky type algorithms, and in particular  when the degree of extensions reaches cryptographic size.
In 2013, Ballet et al. \cite{BBT15} detailed a multiplication algorithm in $\F_{3^{57}}$ with 234 bilinear multiplications using the elliptic curve $y^2+2x^3+2x^2+1=0$ with points of degree at most 4 
and multiplicity at most 3. In 2015, Atighehchi et al. \cite{ABBR15Cras, ABBR15} proposed a model allowing parallel computation, in which an ingenious use of normal bases provides 
efficient algorithms for both multiplication and exponentiation. They detailed an implementation in the finite field $\F_{16^{13}}$ using the hyperelliptic curve $y^2+y=x^5$ of genus 2. In 2012, Riandriambololona \cite{Randriam} 
introduced an asymmetric algorithm which generalizes the Chudnovsky algorithm and leads to better bounds, uniform in $q$, of the bilinearity. When $g=1$, it is known \cite{BBT15} 
that an asymmetric algorithm can always be symmetrized (i.e. there always exists a symmetric version of an asymmetric algorithm). 
However, for greater values of $g$, it may not be the case. Thus, it is of interest to know an effective construction of this asymmetric algorithm.
So far, no effective implementation has been proposed  for such an  algorithm. 
In this article, we detail a strategy to effectively construct  
asymmetric algorithms with higher degree but without derivated evaluation.

\subsection{Organization of the paper and new results}

In Section \ref{GeneRandriam}, we give an explicit translation  of the generalization of the Chudnovsky and Chudnovsky algorithm 
given by Randriambololona \cite[Theorem 3.5]{Randriam}. Then in Section \ref{effect}, by defining a new design of this algorithm, we give a strategy of 
construction and implementation. In particular, thanks to a suitable representation of the Riemann-Roch spaces, we present the first construction of asymmetric effective algorithms 
of multiplication in finite fields. These algorithms are tailored to hardware implementation and they allow computations to be parallelized while maintaining 
a low number of bilinear multiplications. In Section \ref{complex}, we give an analysis of the not asymptotical complexity of this algorithm. Finally, in Sections \ref{effective1}, \ref{effective2}, and \ref{effective3}, 
we give examples with the finite field $\F_{16^{13}}$ using only rational places, $\F_{4^{13}}$ using also places of degree two and $\F_{2^{13}}$ using also places of degree four.

\section{Multiplication algorithms of type Chudnovsky : Generalization of Randriambololona}\label{multalg}\label{GeneRandriam}

In this section we present  a generalization of Chudnovsky type algorithms, introduced in~\cite[Theorem 3.5]{Randriam} by Randriambololona, which is possibly asymmetric. 
Since our aim is to describe explicitly the effective construction of  this asymmetric algorithm, we transform the representation of this algorithm, initially made in the abstract geometrical language, 
in the more explicit language of  algebraic function fields. 
%
A comprehensive course on algebraic function fields  can be found in~\cite{Stich}.
Only the elementary terminology used along this paper is introduced.

Let $F/{\mathbb F}_q$ be an algebraic function field over the finite field ${\mathbb F}_q$
of genus $g(F)$. We denote by $N_1(F/{\mathbb F}_q)$ the number of places of degree one of $F$ over
${\mathbb F}_q$. If $D$ is a divisor, ${\mathcal L}(D)$ denotes the Riemann-Roch space
associated to $D$. We denote by ${\mathcal O}_Q$ the valuation ring of the place $Q$ and
by $F_Q $ its residue class field ${\mathcal O}_Q/Q$
which is isomorphic to $\F_{q^{\deg Q}}$ where $\deg Q$ is the degree of the place $Q$.

In the framework of  algebraic function fields,  
the result \cite[Theorem 3.5]{Randriam} of Randriambololona can be stated as in Theorem~\ref{AlgoRandriam}.
Note that we do not take into account derivated evaluations, since we are not interested in asymptotic results. 
It means that we  describe this asymmetric algorithm with the divisor $G=P_1+ \cdots +P_N$ where the $P_i$ are pairwise disctinct closed points of degree $\deg P_i=d_i$.

Let us define the following Hadamard product in $\F_{q^{l_1}} \times \F_{q^{l_2}} \times \cdots \times \F_{q^{l_N}}$, where the $l_i$'s denote positive integers, by $(u_1,\ldots,u_{N}) \odot (v_1,\ldots,v_{N})=(u_1v_1,\ldots,u_{N}v_{N})$.

\begin{theorem}\label{AlgoRandriam}

Let $F /\F_{q}$ be an algebraic function field of genus $g$ over $\F_q$.
Suppose there exists a place $Q$ of degree $n$. 
Let ${\mathcal P}=\{P_1,\ldots , P_N\}$ be a set of $N$ places of arbitrary degree not containing the place $Q$.
Suppose there exist two effective divisors  $D_1,D_2$ of  $F /\F_{q} $ such that:
\begin{enumerate}
\item[(i)] The place $Q$ and the places of ${\mathcal P}$ are not in the support of the divisors $D_1$ and $D_2$.
\item[(ii)] The natural evaluation maps $E_i$ for $i=1,2$ defined as
$$E_i~:~~\left\{ \begin{array}{ccc}
\mathcal{L}(D_i)  &\longrightarrow &  \F_{q^{n}} \simeq F_Q\   \\ 
f&\longmapsto &f(Q) \end{array}
\right.
$$
are surjective.
\item[(iii)] The natural evaluation  map
$$ 
T~:~~~\left\{ \begin{array}{ccc}
\mathcal{L}(D_1+D_2) &\longrightarrow &\F_{q^{\deg P_1}} \times \F_{q^{\deg P_2}} \times \cdots \times \F_{q^{\deg P_N}}  \\
f&\longmapsto &(f(P_1),f(P_2), \ldots , f(P_N)) \end{array}
\right.
$$
is injective.
\end{enumerate}
Then for any two elements $x,y$ in $\F_{q^n}$, we have:
$$xy=E_Q\circ T^{-1}_{\mid Im \hbox{ }T}\left( T\circ E_1^{-1}(x)\odot T\circ E_2^{-1}(y)  \right),$$ 
where $E_Q$ denotes the canonical projection from the valuation ring ${\mathcal O}_Q$ of the place $Q$ in its residue class field $F_Q$, $\circ$ the standard composition map, 
$T^{-1}_{\mid Im \hbox{ }T}$ the restriction of the inverse map of $T$ on the image of $T$, $E_i^{-1}$ the inverse map of the restriction of the map $E_i$ on the quotient group $\mathcal{L}(D_i)/\ker E_i$ 
and $\odot$ the Hadamard product in $\F_{q^{\deg P_1}} \times \F_{q^{\deg P_2}} \times \cdots \times \F_{q^{\deg P_N}}$; and
$$\mu_{q}(n) \leq \sum_{i=1}^{N} \mu_{q}(\deg P_i).$$

\end{theorem}

\begin{remark} 
The condition $D_1 = D_2$ (modulo the group of principal divisors) is not a sufficient condition to have a symmetric algorithm. 
Indeed, note that even if $D_1=D_2$ then this algorithm can be not symmetric in the sens of Definition (\ref{defalgobil}) if it uses places 
of degree strictly greater than one and if the bilinear multiplications in the residue class fields of these places are not computed (via the operation $\odot$) 
with a symmetric algorithm.
However, for simplicity,  we will say that an algorithm of type Chudnovsky is  symmetric 
when $D_1=D_2$ (modulo the group of principal divisors). 
Indeed, the interest of such an asymmetric algorithm is to avoid the problem of $2$-torsion elements in the divisor class group of the algebraic function field $F/\F_{q}$
in order to have more flexibility in the choice of the algebraic function field $F/\F_{q}$ since the conditions (ii) and (iii) become easier to satisfy (cf. \cite[Remark 3.7]{Randriam}).

Note also that in this presentation, we require that the divisors $D_1$ and $D_2$ are positive divisors in contrast to Theorem 3.5 in \cite{Randriam}.  
Indeed, in our context of effective construction, it is important to have ${\mathcal L}(D_i) \subseteq {\mathcal L}(D_1+D_2)$, which 
 is the case if the divisors $D_1$ and $D_2$ are positive divisors (cf. Remark \ref{remarkeffectivediv}).
 
Finally, for simplicity,  we will use the same representation 
for $\F_{q^{\deg P_i}}$ and $\F_{q^{\deg P_j}}$ when 
$\deg P_i=\deg P_j$. Actually, we  could only suppose 
 that the products in $\F_{q^{\deg P_i}}$ and $\F_{q^{\deg P_j}}$ use  bilinear multiplication algorithms having the same bilinear complexity.
 
\end{remark}

\section{Effective algorithm}\label{effect}

\subsection{Method and strategy of implementation}\label{construction}
The construction of the algorithm is based on the choice of the place $Q$ of degree $n$,  
the effective divisors $D_1$ and $D_2$  of degree $n+g-1$, the bases of spaces ${\mathcal L}(D_1) $, ${\mathcal L}(D_2) $
 and ${\mathcal L}(D_1+D_2)$ and 
the basis of the residue class field $F_Q$  of the place $Q$. 

In practice, following the ideas of  \cite{Ballet2}, we take as a divisor $D_1$ one place of degree $n+g-1$.
This has the advantage to solve both the problem of the support of divisor $D_1$ 
and the problem of the effectivity of the divisor $D_1$. For the same reasons, the divisor $D_2$ is also
chosen as a place of degree $n+g-1$. 
Furthermore, we require additional properties described below.

\subsection{Finding good places $D_1$, $D_2$ and $Q$}\label{si}

In order to obtain the good places, we draw them at random and   check that they satisfy the required 
conditions. We proceed as follows:

\begin{enumerate}\label{placeswithgoodproperties}
\item We draw at random an irreducible polynomial $\mathcal{Q}(x)$ of degree $n$ in $\F_q[X]$ and check that this polynomial is:
     \begin{enumerate}
           \item Primitive.
           \item Totally decomposed in the algebraic function field $F/\F_q$ (which implies that there exists a place $Q$ of degree n  
           above the polynomial $\mathcal{Q}(x)$).           
     \end{enumerate}
     \item We choose a place $Q$ of degree $n$ 
     among the places of $F/\F_q$  lying above the polynomial $\mathcal{Q}(x)$. 
     \item We draw at random a place $D_1$ of degree $n+g-1$ and check that $D_1-Q$ is a non-special divisor of degree $g-1$ 
     i.e. $\dim {\mathcal L}(D_1-Q)=0$.
     \item We draw at random a place $D_2$ of degree $n+g-1$ and check that $D_2-Q$  is a non-special divisor of degree $g-1$ i.e. $\dim(D_2-Q)=0$.
    \end{enumerate}

%

\begin{remark}\label{remarkeffectivediv}
Clearly, our method relies on the existence of places $Q$, $D_1$ and $D_2$ as defined above and
such that $D_1 - Q$ and $D_2 - Q$ are non-special divisor of degree $g-1$.
We say some words about this.

On the place $Q$, a sufficient condition for the existence of at least one place of degree $n$ is given by the following inequality by \cite[Corollary V.2.10 (c)]{Stich}:
$$2g+1 \leq q^{\frac{n-1}{2}}\left(q^{\frac{1}{2}}-1\right).$$

Then, we are sure of the existence of a non-special divisor of degree $g-1$ when $q\geq 4$  \cite{BalletBrigand}.
The larger $q$ is, the larger  the probability to draw a non-special divisor of degree $g-1$ becomes (Proposition 5.1 \cite{BalletRollandRitzenthaler}),
but not necessarily as a difference of two places: this is an open problem. 
However, looking for non-special divisors of degree $g-1$ as a difference of two places
has many advantages.

Most of all, this solves easily 
the problem of the support of divisors $D_1$ and $D_2$ (condition (i) of Theorem 
\ref{AlgoRandriam}) as well as the problem of the effectivity of these divisors.
Indeed, in our context of construction, it is important to have 
${\mathcal L}(D_i) \subseteq {\mathcal L}(D_1+D_2)$, which 
 is the case if the divisors $D_1$ and $D_2$ are effective divisors. 
 However, the property to have simultaneously an effective divisor 
 (with the required properties) without having given places in its support is difficult to obtain theoretically because the method 
 of the support moving (cf. \cite{piel}), 
 which is a direct consequence of Strong Approximation Theorem (cf.  \cite[Proof of Theorem I.6.4]{Stich}), 
 has the drawback to imply the loss of effectivity.

Furthermore, in practice, it is easy to find $Q$ and the divisors $D_i$ satisfying the required properties since there exist many such places  in our context. 
However, it is not true in the  general case. 
For instance, this fails when we consider an elliptic curve with only one rational point since for any elliptic curve,
there exists a non-special divisor of degree $g-1=0$ if and only if the divisor class number $h$ is $>1$, i.e. $N_1\geq 2$ 
(cf. \cite[Section 3.2]{BalletBrigand}). 

\end{remark}

\subsection{Choosing good bases of the spaces } \label{choixbases}

\subsubsection{The residue field $F_Q$}\label{baseFQ}

When we take a place $Q$ of degree $n$ lying above a  polynomial in $\F_q[X]$, we mean that the residue class field 
is the finite field $\F_{q^n}$ for which we choose as a representation basis  the  canonical basis $\mathcal B_Q$ generated 
by a root $\alpha$ of the polynomial $\mathcal{Q}(x)$, namely ${\mathcal B}_Q=( 1, \alpha, \alpha^2,..., \alpha^{n-1})$. 
Note that if we wish to use a normal basis as a representation basis, it is convenient to find a place $Q$ above a normal polynomial 
$\mathcal{Q}(x)$ and to make a change of basis between the canonical basis and the normal basis $( \alpha, \alpha^{q}, \alpha^{q^2}...,\alpha^{q^{n-1}})$. 
However, even in this case, it is necessary in our algorithm to preserve the canonical basis as basis of the residue field $F_Q$ 
because we need to have the constant component in the bases of the Riemann-Roch spaces 
${\mathcal L}(D_i)$ for $i\in \{1,2\}$ (cf. Section \ref{baseLD1}).
From now on we identify $\F_{q^n}$ to $F_Q$, 
as the residue class field $F_Q$ of the place $Q$ is isomorphic to the finite field $\F_{q^n}$.

\subsubsection{The Riemann-Roch spaces $\mathcal{L}(D_1)$ and $\mathcal{L}(D_2)$}\label{baseLD1}

Clearly, the choice of $D_i$, $i\in \{1,2\}$ and $Q$ of Section \ref{si} implies that  the maps $E_i$ of Theorem~\ref{AlgoRandriam} are isomorphisms, since
 $\deg (D_i)=n+g-1$,  $\dim {\mathcal L}(D_i-Q)=0$ and ${\mathcal L}(D_i-Q)=Ker(E_i)$.
Thereby, we choose as basis of ${\mathcal L}(D_i)$ the reciprocal image $\mathcal{B}_{D_i}$ of the basis $\mathcal{B}_Q=(\phi_1,\ldots,\phi_n)$ of  $F_Q$ 
by the evaluation map $E_i$, namely $\mathcal{B}_{D_i}=(E_i^{-1}(\phi_i),\ldots,E_i^{-1}(\phi_n))$. Note that by Section \ref{baseFQ}, the choice of the basis of 
the residue field $F_Q$ implies that $\phi_1=1$ and so $E_1^{-1}(1)=E_2^{-1}(1)=1$. Let us denote $\mathcal{B}_{D_i}=(f_{i,1},...,f_{i,n})$ with $f_{i,1}=1$ for $i=1,2$.

\subsubsection{The Riemann-Roch space $\mathcal{L}(D_1+D_2)$} \label{baseLD1D2}

Note that since  $D_1$ and $D_2$ are effective divisors, we have ${\mathcal L}(D_1)\subset {\mathcal L}(D_1+D_2)$ and ${\mathcal L}(D_2)\subset {\mathcal L}(D_1+D_2)$. 

\begin{lemma}\label{intersectioneffectivedivisor}
Let $D_1$ and $D_2$ be two effective divisors with disjoint supports. 
Then 
$${\mathcal L}(D_1)\cap {\mathcal L}(D_2)= \F_q.$$
\end{lemma}

\begin{proof}
It is clear that $\F_q\subset {\mathcal L}(D_1)\cap {\mathcal L}(D_2)$ because the divisors are effective.
Suppose that the function $f \in {\mathcal L}(D_1)\cap {\mathcal L}(D_2)$ is such that $f\notin \F_q$. Then there exist $P_1$ in the support of $D_1$ 
and $P_2$ in the support of $D_2$ such that $v_{P_1}(f)\leq -1$ and $v_{P_2}(f)\leq -1$. But then the function $f$ admits a pole of order at least 1 at $P_1$ 
and a pole  of order at least 1 at $P_2$, which is impossible because the supports are disjoint. So, $f\in \F_q$ and the proof is complete.
\end{proof}

\begin{proposition}\label{supplementaryspaces}
Let $D_1$, $D_2$ and $Q$ be places having the properties described in (\ref{placeswithgoodproperties}).
Consider the map $\Lambda : {\mathcal L}(D_1+D_2) \rightarrow F_Q$ such that $\Lambda (f)=f(Q)$ for $f\in {\mathcal L}(D_1+D_2)$.
There exists a vector space $ {\mathcal M}\subseteq \ker \Lambda$ of dimension $g$ such that 
$${\mathcal L}(D_1+D_2)={\mathcal L}(D_1) \oplus  {\mathcal L}_r(D_2) \oplus {\mathcal M},$$ where  ${\mathcal L}_r(D_2)$ is such that 
$${\mathcal L}(D_2)= \F_q \oplus {\mathcal L}_r(D_2)$$ and $\oplus$ denotes 
the direct sum. In particular, if $g=0$, then ${\mathcal M}= Ker\Lambda$ is equal to $\{0\}$.
\end{proposition}

\begin{proof}
The divisors $D_1$ and $D_2$ which are effective divisors of degree $n+g-1$ have the same dimension $n$ by Section \ref{baseLD1}. 
Moreover, as $\deg (D_1+D_2 )> 2g-2$, the divisor $D_1+D_2$ is non-special and $\dim {\mathcal L}(D_1+D_2)=2n+g-1$.
Hence, $\dim Ker\Lambda= n+g-1$ and by Lemma \ref{intersectioneffectivedivisor}, the spaces ${\mathcal L}(D_1)$ and  
${\mathcal L}_r(D_2)$ consist on a direct sum of dimension $\dim \biggr( {\mathcal L}(D_1) \oplus  
{\mathcal L}_r(D_2)\biggl)=2n-1$. 
In Section  \ref{baseLD1}, we have showed that $Ker(E_i)=\{0\}$. Thus, the set ${\mathcal L}(D_i)\cap Ker\Lambda=\{0\}$ as well.
Hence, there exists a vector space ${\mathcal M}\subseteq Ker\Lambda$ of dimension $g$ 
which gives the result. 
\end{proof}


Hence, we choose as basis of ${\mathcal L}(D_1+D_2)$ the basis $\mathcal{B}_{D_1+D_2}$ defined by: 

$$\mathcal{B}_{D_1+D_2}=(f_1, \ldots ,f_n,f_{n+1},\ldots , f_{2n+g-1})$$
where $\mathcal{B}_{D_1}=(f_1, \ldots ,f_n)$ is the basis of ${\mathcal L}(D_1)$,
$(f_{n+1},\ldots , f_{2n-1})$ is a basis of ${\mathcal L}_r(D_2)$ such that $f_{n+j}=f_{2,j+1}\in \mathcal{B}_{D_2}$ with $\mathcal{B}_{D_1}$ and $\mathcal{B}_{D_2}$ 
defined in Section \ref{baseLD1} and $\mathcal{B}_{{\mathcal M}}=(f_{2n},\ldots , f_{2n+g-1})$ is a basis  of ${\mathcal M}$.

\subsection{Product of two elements in ${\mathbb F}_{q^n}$}

In this section, we use as representation bases of spaces $F_Q$, ${\mathcal L}(D_i)$ $(i\in\{1,2\})$, ${\mathcal L}(D_1+D_2)$, 
the bases defined in Section \ref{choixbases}. The product of two elements in ${\mathbb F}_{q^n}$ 
is computed by the algorithm of Chudnovsky and Chudnovsky.
Let $x=(x_1,\ldots,x_n)$ and $y=(y_1,\ldots,y_n)$ be 
two elements of ${\mathbb F}_{q^n}$ given by their components over ${\mathbb F}_{q}$
relative to the chosen basis $\mathcal{B}_Q$. 
According to the previous notation, we can consider that $x$ and $y$ are 
identified to the following elements:

$$f_x=\sum_{i=1}^n x_if_{1,i} \in {\mathcal L}(D_1)\quad \hbox{and} \quad f_y=\sum_{i=1}^n y_if_{2,i}\in {\mathcal L}(D_2).$$

The product $f_xf_y$ of the two elements $f_x$ and $f_y$  is their product in 
the valuation ring ${\mathcal O}_Q$. This product lies in ${\mathcal L}(D_1+D_2)$ since $D_1$ and $D_2$ are effective divisors.
We consider that $x$ and $y$ are respectively the elements
$f_x$ and $f_y$ embedded in the Rieman-Roch space ${\mathcal L}(D_1+D_2)$, via respectively the embeddings 
$I_i: {\mathcal L}(D_i) \longrightarrow  {\mathcal L}(D_1+D_2)$ defined by $I_1(f_x)$ and  $I_2(f_y)$ as follows. 
If, $f_x$ and $f_y$ have respectively coordinates $f_{x_i}$ and $f_{y_i}$  in $\mathcal{B}_{D_1+D_2}$ where $i\in\{1, \ldots , 2n+g-1\}$, we have:
$I_1(f_x)=(f_{x_1}:= x_1,\ldots , f_{x_n}:=x_n, 0,\ldots , 0)$ and $I_2(f_y)=( f_{x_1}:=y_1, 0, \ldots , 0, f_{y_{n+1}}:=y_2,\ldots , f_{y_{2n-1}}:=y_{n}, 0, \ldots 0)$.
 Now it is clear that knowing $x$ (resp. $y$) or
$f_x$ (resp. $f_y$) by their coordinates is the same thing.

\begin{theorem} 
Let $P_{{\mathcal M}^s}$ be the projection of ${\mathcal L}(D_1+D_2)$ 
onto ${\mathcal M}^s={\mathcal L}(D_1) \oplus  {\mathcal L}_r(D_2)$ and let $\Lambda$ be the map defined as in Proposition (\ref{supplementaryspaces}).
Then, for any elements $x,y\in \F_{q^n}$, the product of $x$ by $y$ is such that
$$xy=\Lambda\circ P_{{\mathcal M}^s}\left(T^{-1}_{\mid Im \hbox{ }T} \left(\strut T\circ I_1\circ E_1^{-1}(x)\odot T\circ I_2\circ E_2^{-1}(y)\right)\right),$$ where $\circ$ denotes 
the standard composition map, $T^{-1}_{\mid Im \hbox{ }T}$ the restriction of the inverse map 
of $T$ on the image of $T$, and $\odot$ the Hadamard product as in Theorem \ref{AlgoRandriam}.
\end{theorem}

\begin{proof}
Let $\F_{q^n}$ be a finite extension of $\F_q$ of degree $n$ with a representation defined in Section \ref{baseFQ}. For any two elements $x,y\in \F_{q^{n}}$, 
there exist two elements $f_x$ and $f_y$ respectively in ${\mathcal L}(D_1)$ and ${\mathcal L}(D_2)$ defined as in Section \ref{baseLD1}
such that $E_1(f_x)=f_x(Q)=x$ and $E_2(f_y)=f_y(Q)=y$ where $f(Q)$ denotes the class of $f$ in the residue class field $F_Q$ of the place $Q$. 
Thus, $$xy=f_x(Q)f_y(Q)=(f_xf_y)(Q)$$ and so computing the product $xy$ is equivalent to computing the product $f_xf_y$.
Moreover, since the divisors $D_1$ and $D_2$ are effective by the assumptions of Theorem \ref{AlgoRandriam} and Section \ref{construction}, we have
 $<{\mathcal L}(D_1){\mathcal L}(D_2)>\subset {\mathcal L}(D_1+D_2)$ where $<{\mathcal L}(D_1){\mathcal L}(D_2)>$ denotes the vector space generated by 
 the products $f_xf_y$ with $f_x\in {\mathcal L}(D_1)$,  $f_y\in {\mathcal L}(D_2)$ and so $f_xf_y\in {\mathcal L}(D_1+D_2)$.
Now, the principle of the algorithm is to compute $f_xf_y$ via the evaluation map $T$. In this aim, we represent the elements $f_x$ and 
$f_y$ in ${\mathcal L}(D_1+D_2)$ respectively by $I_1(f_x)$ and $I_2(f_y)$ defined in this section. Then by Theorem \ref{AlgoRandriam}, 
$$h=f_xf_y= T^{-1}_{\mid Im \hbox{ }T}\left(T\circ I_1(f_x)\odot T\circ I_2(f_y)\right) \hbox{ and }h(Q)=xy.$$ 
Then, the proof is complete since $h(Q)=\Lambda \circ P_{{\mathcal M}^s}(h)$.
\end{proof}

We can now present the setup algorithm and the multiplication algorithm.
Note that the setup algorithm is only done once.

\begin{algorithm}[H] 
\caption{Setup algorithm}
\begin{algorithmic}
\Require $F/{\mathbb F}_{q}, ~  Q,  D_1, D_2, P_1,\ldots, P_{N}$.
\Ensure  $T \hbox{ and } T^{-1}.$
        \begin{enumerate}
             \item The representation of the finite field $\F_q=<a>$, where $a$ is a primitive element i.e. a generator of the associated cyclic group, is fixed.
             \item The function field $F/{\mathbb F}_{q}$, the place $Q$, the
divisors $D_1$ and $D_2$ and the points $P_1,\ldots, P_{N}$ are such that Conditions (ii) and (iii) in Theorem \ref{AlgoRandriam} are satisfied.
In addition, we require that $\sum_{1\le i \le N}\deg P_i = 2n+g-1$.
          \item Represent $\F_{q^n}$ in the canonical basis ${\mathcal B}_Q=\{ 1, \alpha, \alpha^2,..., \alpha^{n-1}\}$, where $\F_{q^n}=<\alpha>$ with $\alpha$ a primitive element 
             as in Section \ref{baseFQ}.
          \item Construct a basis
$(f_1, \ldots, f_n,f_{n+1}, \ldots, f_{2n+g-1})$ of ${\mathcal L}(D_1+D_2)$ where
$(f_1, \ldots, f_n)$ is the basis of ${\mathcal L}(D_1)$,
$(f_1, f_{n+1}, \ldots,  f_{2n-1})$ the basis of ${\mathcal L}(D_2)$ and $(f_{2n}, \ldots, f_{2n+g-1})$ the basis of ${\mathcal M}$, defined in Section \ref{baseLD1}.
          \item Compute the matrices $T$ and $T^{-1}$.
          \item Compute the matrice $\Lambda$. 
        \end{enumerate}
\end{algorithmic}
\end{algorithm}

Note that any element $z$ of the field ${\mathbb F}_{q^n}$ is known by its components relatively to the canonical basis ${\mathcal B}_Q$:
 $z=(z_1,\ldots,z_n)\in \F_{q^n}$ (where $z_i \in {\mathbb F}_{q}$). Then, we have two ways to represent $z$ in our algorithm: 
 embedding $z$ in ${\mathcal L}(D_1+D_2)$ via $I_1\circ E_1^{-1}$ or via $I_2\circ E_2^{-1}$. When we want to multiply two elements $x$ and $y$, 
 we choose conventionally to represent $x$ by $I_1\circ E_1^{-1}(x)$ and $y$ by $I_2\circ E_2^{-1}(y)$.

\begin{algorithm}[H]
\caption{Multiplication algorithm} \label{Algo2}
\begin{algorithmic}
\Require $x=(x_1,\ldots,x_n) \hbox{ and } y=(y_1,\ldots,y_n)$.
\Ensure  $xy.$
       \begin{enumerate}
         \item Compute 
$$\left (\begin{array}{c} z_{1,d_1}\\ \vdots\\ z_{n,d_n}\\z_{n+1,d_{n+1}} \\ \vdots\\ z_{N,d_N}
\end{array} \right)=
\left (\begin{array}{c} z_1\\ \vdots\\ z_n\\z_{n+1} \\ \vdots\\ z_{2n+g-1}
\end{array} \right)= T\left (\begin{array}{c} x_1\\ \vdots\\ x_n\\ 0\\ \vdots\\
0 \end{array} \right)  \hbox{ and } 
\left (\begin{array}{c} t_{1,d_1}\\ \vdots\\ t_{n,d_n}\\t_{n+1,d_{n+1}} \\ \vdots\\ t_{N,d_N}
\end{array} \right)=
\left (\begin{array}{c} t_1\\ \vdots\\t_n\\t_{n+1}\\ \vdots\\ t_{2n+g-1}
\end{array} \right)= T\left (\begin{array}{c} y_1\\0\\ \vdots\\ 0\\ y_2\\ \vdots\\ y_n\\0\\ \vdots\\ 0 
\end{array} \right).
$$ 
where $\sum_{i=1}^{N}d_i=2n+g-1$ and $(z_{i,j}, t_{i,j})\in (\F_{q^{d_j}})^2$ and $(z_i,t_i)\in (\F_q)^2$.
        \item Compute the Hadamard product $u=(u_{1,d_1},\ldots, u_{N,d_N})= (u_1,\ldots, u_{2n+g-1})$, where $u_{i,d_i}=z_{i,d_i}t_{i,d_i}$, in $\F_{q^{d_1}} \times \F_{q^{d_2}} \times \cdots \times \F_{q^{d_N}}$
        as in Theorem \ref{AlgoRandriam}.
        \item Compute $w=(w_1,\ldots,w_{2n+g-1})=T^{-1}(u)$.
        \item Extract $w'=(w_1,\ldots,w_{2n-1})$ (remark that in the previous step
we just have  to compute the $2n-1$ first components of $w$).
        \item Return xy=$\Lambda(w')$.  
       \end{enumerate}
\end{algorithmic}
\end{algorithm}

\section{Complexity analysis}\label{complex}

By Theorem \ref{AlgoRandriam}, the condition (ii) implies $\dim {\mathcal L}(D_i)\geq n$. But  by Riemann-Roch Theorem, $\dim {\mathcal L}(D_i)\geq -g+1+ \deg D_i$ 
which gives in the least case: $\deg D_i \geq n+g-1$ and so, $\deg  (D_1+D_2) \geq 2n+2g-2$. Without loss of generality, we suppose that 
$\deg D_i = n+g-1$ and so $\deg  (D_1+D_2) = 2n+2g-2$. 
%
In this case, $\deg  (D_1+D_2)\ge 2g-1$ and then we obtain $\dim {\mathcal L}(D_1+D_2) = 2n+g-1$.
According to Theorem~\ref{GeneRandriam}, 
Algorithm~\ref{Algo2} requires that the natural evaluation map $T$ is injective.
A sufficient condition to get injectivity is given by Condition~\ref{conditionnbrepoints}:
\begin{equation}\label{conditionnbrepoints}
\sum_{i\mid r}^{r} iN_i> 2n+2g-2
\end{equation} 
where $N_i$ denotes the number of places of degree $i$ in ${\mathcal P}$ and $r$ an integer $>1$.
Indeed, the kernel of $T$ is ${\mathcal L}(D_1+D_2 - \sum_{P\in \mathcal{P}} P)$ and,
under Condition~\ref{conditionnbrepoints}, this kernel is trivial
since the divisor $D_1+D_2 - \sum_{P\in \mathcal{P}} P$ has negative degree.

In terms of number of multiplications in $\F_q$, the complexity of this multiplication algorithm is as follows:
calculation  of $z$ and $t$ needs $2 (2n^2+ng-n)$ multiplications, calculation of $u$ needs $(2n+2g-2+r)\sup_{1\leq i \leq r}\frac{\mu_q(i)}{i}$ bilinear multiplications 
and calculation of $2n-1$ first components of $w$ needs $(2n+g-1)(2n-1)$ multiplications (remark that in Algorithm \ref{Algo2}, 
we just have  to compute the $2n-1$ first components of $w$).  
The calculation of $xy$ needs $n+g$ multiplications. The total 
number of multiplications is bounded by 
$8n^2+n(4g-5)+(2n+2g-2+r)\sup_{1\leq i \leq r}\frac{\mu_q(i)}{i}$.

The asymptotic analysis of our method needs to consider infinite families of algebraic 
function fields defined over $\F_q$ with increasing genus 
(or equivalently of algebraic curves) having the required properties.
The existence of such families follows from that of families of algebraic function fields reaching the Generalized 
Drinfeld-Vladut bound of order $r$ (cf. \cite{Ballet6}). For example,  it is proved in \cite{Ballet2} (with rational places i.e. $r=1$), \cite{Ballet4} and 
in \cite{Ballet7} (with places of degree two i.e. $r=2$), in \cite{BalletPieltant} (with places of degree four i.e. $r=4$)
from a specialization of the type Chudnovsky symmetric algorithms on recursive towers of algebraic function fields of type 
Garcia-Stichtenoth that the bilinear complexity of the multiplication in any degree $n$ extension of $\F_q$ 
is uniformly linear in $q$ with respect to $n$. Good asymptotic bounds are also obtained by using families of modular Shimura curves \cite{Ballet8}. 
Similarly, Randriambololona improved the uniform (resp. asymptotic) bounds with an asymmetric algorithm of type Theorem \ref{AlgoRandriam}. 
Hence, the number of bilinear multiplications of the algorithm \ref{AlgoRandriam} is in $O(n)$ when the places used in the algorithm \ref{AlgoRandriam} 
have a degree one or two. Moreover, the genus $g$ of the required curves also necessarily increases in $O(n)$. 
Consequently, the total number of multiplications/additions/subtractions of the algorithm \ref{AlgoRandriam}  
is in $\BigO{n^2}$ and the total number of bilinear multiplications is in $\BigO{n}$.


 \section{Multiplication in $\F_{16^{n}/\F_{16}}$}\label{effective1}
 
 Set $q=16$ and $n=13, 14, 15$. Note that the multiplication algorithms in the extensions of degree $n<13$ are 
 symmetric because they are obtained with rational and elliptic function fields (with the best possible bilinear complexities of multiplication). 
 Hence, it is only pertinent to consider the multiplication in extensions of degree $\geq 13$ namely with function fields of genus $g\geq 2$. 
 From now on, $F/\F_q$ denotes the algebraic function field
 associated to the hyperelliptic curve $X$ with plane model $y^2+y=x^5$, of genus two. This curve has 33 
 rational points, which is maximal over $\F_q$ according to the Hasse-Weil bound. We represent $\F_{16}$  
 as the field $\F_2(a)=\F_2[X]/(P(X))$ where $P(X)$ is the irreducible polynomial $P(X)= X^4+X+1$ and 
 $a$ denotes a primitive root of $P(X)= X^4+X+1$. Let us give the projective coordinates $(x:y:z)$ of 
 rational points of the curve $X$:
 
 \vspace{.5em}
 
 $$
 \begin{array}{lll}
 P_{\infty}=(0: 1 : 0) & P_{2}=(0: 0: 1) & P_{3}=( 0: 1: 1) \\
 P_{4}=(a : a: 1)& P_{5}=( a: a^4: 1) & P_{6}=(a^2: a^2: 1) \\
 P_{7}=( a^2: a^8: 1) & P_{8}=( a^3: a^5: 1)& P_{9}=( a^3: a^{10}: 1) \\
 P_{10}=(a^4: a: 1) & P_{11}=( a^4: a^4: 1) & P_{12}=( a^5: a^2: 1) \\
 P_{13}=(a^5 : a^8: 1) & P_{14}=(a^6: a^5: 1) & P_{15}=( a^6: a^{10}: 1) \\
 P_{16}=( a^7: a: 1) & P_{17}=( a^7: a^4: 1)& P_{18}=(a^8: a^2: 1) \\
 P_{19}=( a^8: a^8: 1) & P_{20}=( a^9: a^5: 1) & P_{21}=( a^9:a^{10} : 1) \\
 P_{22}=(a^{10}: a: 1) & P_{23}=( a^{10}: a^4: 1) & P_{24}=( a^{11}: a^2: 1) \\
 P_{25}=( a^{11}:a^8 :1 ) & P_{26}=(a^{12}:a^5 : 1) &P_{27}=( a^{12}: a^{10}: 1) \\
 P_{28}=( a^{13}: a: 1) & P_{29}=( a^{13}: a^4: 1) & P_{30}=(a^{14}: a^2: 1) \\
 P_{31}=( a^{14}: a^8: 1) & P_{32}=( 1: a^5: 1) & P_{33}=( 1: a^{10}: 1) \\
 \end{array}
 $$
 
 \subsection{Construction of the required divisors}\label{SS-constructionOfDivisors}
 
 \subsubsection{A place $Q$ of degree n}\label{Q13}
 
 It is sufficient to take a place $Q$ of degree $n$ in the rational function field $\F_{q}(x)/\F_q$, 
 which totally splits in $F/\F_q$. It is equivalent to choose a monic irreducible polynomial 
 $\mathcal{Q}(x)\in \F_{q}[x]$ of degree n such that its roots $\alpha_i$ in $\F_{q^n}$ satisfy $Tr_{\F_{2}}(\alpha_i^5)=0$ 
 for $i=1,...,n$ where the map $Tr_{\F_{2}}$ denotes the classical function Trace over $\F_{2}$ by \cite[Theorem 2.25]{LidlNi}.
 In fact, it is sufficient to verify that this property is satisfied for only one root since a finite field is Galois.

 \vspace{1em}
 
 For example, for the extension $n=13$, we choose the irreducible polynomial 
 
 \begin{equation}
 \begin{aligned}
 \mathcal{Q}(x)=\; & x^{13}+a^{6}x^{12}+a^{5}x^{11}+a^{11}x^{10}+x^{9}+a^{12}x^{8}+\\
 & a^{7}x^{7}+a^{7}x^{5}+a^{2}x^{4}+a^{11}x^{3}+a^{8}x^{2}+a^{6}x+a^{14} .
 \end{aligned}
 \end{equation}
 
 Let $b$ be a root of  $\mathcal{Q}(x)$. It is easy to check that $Tr_{\F_{2}}(b^5)=0$, hence 
 the place $(\mathcal{Q}(x))$ of $\F_{16}(x)/\F_{16}$ 
 is totally splitted in the algebraic function field $F/\F_q$, which means that there exist 
 two places of degree $n$ in $F/\F_q$ lying over the place $(\mathcal{Q}(x))$ of 
 $\F_{16}(x)/\F_{16}$. 
  For the place $Q$ of degree $n$ in the algebraic function field $F/\F_q$, we consider one of the two 
  places in $F/\F_q$ lying over the place 
 $(\mathcal{Q}(x))$ of $\F_{16}(x)/\F_{16}$, namely the orbit of the $\F_{16^{13}}$-rational point 
 $\mathcal{P}_{1i}=(\alpha_i: \beta_i:1)$ where $\alpha_i$ 
 is a root of $\mathcal{Q}(x)$ and 
 $\beta_i= a^6\alpha_i^{12} + a^{13}\alpha_i^{11} + a\alpha_i^{10} + a^{13}\alpha_i^9 + a^8\alpha_i^8 + a\alpha_i^7 + a^8\alpha_i^6 +
   a^9\alpha_i^5 + a^5\alpha_i^4 + a^2\alpha_i^2 + a^8\alpha_i + a^{13}$ for $i=1,...,13$. 
   Notice that the second place is given by the conjugated points 
   $\mathcal{P}_{2i}=(\alpha_i: \beta_i+1:1)$ for $i=1,...,13$.
 
  \vspace{1em}
  
   \subsubsection{The two divisors $D_1$ and $D_2$ of degree n+g-1}
 
 For the divisor $D_1$ of degree $n+g-1$, we choose a place $D_1$ of degree 14 according to the method used for the place $Q$. 
 We consider the orbit of the $\F_{16^{14}}$-rational point $\mathcal{P}^1_{1i}=(\gamma_i: \delta_i:1)$ where $\gamma_i$ 
 is a root of $\mathcal{D}_1(x)=x^{14} + a^9x^{13} + a^6x^{12} + a^7x^{11} + a^{11}x^{10} + a^{12}x^9 + a^{10}x^8 +
   a^6x^7 + a^7x^6 + a^{10}x^5 + a^{14}x^4 + x^3 + x^2 + a^3x + a$
 and $\delta_i=a^4\gamma_i^{12} + a^8\gamma_i^{11} + a^7\gamma_i^9 + a^2\gamma_i^8 + a^3\gamma_i^7 + a^8\gamma_i^6 
+ a^4\gamma_i^5 +a^{14}\gamma_i^4 + \gamma_i^2 + a^6\gamma_i + a^3$ for $i=1,...,14$.
 Notice that the second place is given by the conjugated points $\mathcal{T}^1_{2i}=(\gamma_i: \delta_i+1:1)$ for $i=1,...,14$.
 
  \vspace{1em}
  
  
   For the divisor $D_2$ of degree $n+g-1$, we choose a place $D_2$ of degree 14 according to the method used for the place $Q$. 
 We consider the orbit of the $\F_{16^{14}}$-rational point $\mathcal{P}^2_{1i}=(\gamma_i: \delta_i:1)$ where $\gamma_i$ 
 is a root of $\mathcal{D}_2(x)=x^{14} + x^2 + ax + 1 $
 and  $\delta_i=a^5\gamma_i^{12} + a^{11}\gamma_i^{11} + a^{11}\gamma_i^{10} + a^8\gamma_i^9 + a^4\gamma_i^8 + a^8\gamma_i^7 + \gamma_i^6 + a^8\gamma_i^5 
    + a^2\gamma_i^4 + a^9\gamma_i^3 + a^2\gamma_i^2 + a^2\gamma_i + a^7$
  for $i=1,...,14$.
 Notice that the second place is given by the conjugated points $\mathcal{T}^2_{2i}=(\gamma_i: \delta_i+1:1)$ for $i=1,...,14$.
 
  \vspace{1em}
  
  The place $Q$ and the divisors $D_1$ and  $D_2$ satisfy the good properties since the dimensions of the divisor $D_1-Q$ and  $D_2-Q$ are 
  zero which means that the divisors $D_1-Q$ and  $D_2-Q$ are non-special of degree $g-1$.

\subsection{Construction of required bases}
  
  \subsubsection{The basis of the residue class field $F_Q$}
  
  We choose as basis of the residue class field $F_Q$ the  basis ${\mathcal B}_Q$ associated to 
  the place $Q$ obtained in Section \ref{Q13}.

  \subsubsection{The basis of $\mathcal{L}(D_i)$ for i=1,2}
  
  We choose as basis of the Riemann-Roch space $\mathcal{L}(D_i)$ the basis ${\mathcal B}_{D_i }$
  such that $E_i({\mathcal B}_{D_i})={\mathcal B}_Q$ 
  is a basis of $F_Q$ as in Section \ref{baseLD1}, 
  ${\mathcal B}_{D_1}=(f_1,...,f_n)$ and  ${\mathcal B}_{D_2}=(f_1,f_{n+1}...,f_{2n-1})$. 
  For $j\in \{2, \ldots , n\}$, any element $f_j$ of ${\mathcal B}_{D_1}$  is such that $$f_j(x,y)=\frac{f_{j1}(x)y+f_{j2}(x)}{\mathcal D_1(x)},$$ and for $j\in \{n+1, \ldots , 2n-1\}$, any element $f_j$ of ${\mathcal B}_{D_2}$ is such that $$f_j(x,y)=\frac{f_{j1}(x)y+f_{j2}(x)}{\mathcal D_2 (x)},$$ 
  where $f_{j1}, f_{j2}\in \F_{16}[x]$. To simplify, we set $f_j(x,y)=(f_{j1}(x), f_{j2}(x))$. 
 We have:
  
  \vspace{1em}
  
$f_1(x,y)=1$,\\

$f_2(x,y)=( a^{13} x^{11} + a^{10} x^{10} + a^{3} x^{9} + a^{10} x^{8} + a^{14} x^{7} + a^{11} x^{6} + a^{8} x^{5} + a^{11} x^{4} + x^{3} + a x^{2} +
        a^{11} x + a^{11},\; a^{12} x^{14} + a^{12} x^{13} + a^{9} x^{12} + 
        x^{11} + a^{8} x^{10} + a^{13} x^{9} + a^{12} x^{8} + a x^{7} + a^{5} x^{6} + x^{5} + a^{13} x^{4} + a^{5} x^{3} + a^{12} x^{2} + a^{4} x )$,\\
   
$f_3(x,y)=( a^{2} x^{11} + a^{11} x^{10} + a^{13} x^{9} + a^{3} x^{8} + a^{10} x^{7} + a x^{6} + a^{9} x^{5} + a^{6} x^{4} + a^{5} x^{3} + a^{3} x^{2} + a^{7} x + a^{14}  ,\;
 a^{12} x^{14} + a^{2} x^{13} + a^{11} x^{12} + 
        a^{3} x^{11} + a^{13} x^{10} + a^{7} x^{9} + a^{9} x^{8} + a^{9} x^{7} + a^{13} x^{6} + a x^{5} + a^{5} x^{4} + a^{12} x^{3} + 
        a^{13} x^{2} + a^{8} x + a^{2}  )$,\\
        
$f_4(x,y)=( a^{8} x^{11} + a x^{10} + a^{10} x^{9} + x^{8} + a^{8} x^{7} + a^{14} x^{6} + a^{6} x^{5} + a^{3} x^{4} + a^{14} x^{3} + a^{3} x^{2} + 
        a^{6} x + a   ,\;
        a^{3} x^{14} + a^{7} x^{13} + a^{10} x^{12} + 
        a^{11} x^{11} + a^{13} x^{9} + a^{8} x^{8} + a^{8} x^{7} + a^{3} x^{6} + a^{5} x^{5} + a^{6} x^{4} + a^{3} x^{3} + x^{2} + a^{5} x + 
        a^{6}
    )$,\\  
   
$f_5(x,y)=(  a^{12} x^{11} + a^{12} x^{10} + a^{14} x^{9} + a^{7} x^{8} + a^{7} x^{7} + a^{7} x^{6} + a^{3} x^{5} + a^{13} x^{4} + a^{2} x^{3} + 
        a^{7} x^{2} + a^{7} x + a^{7}  ,\;
        a^{2} x^{14} + a^{11} x^{13} + 
        a^{5} x^{12} + a^{10} x^{11} + a^{10} x^{9} + a^{13} x^{8} + a x^{7} + a^{10} x^{6} + a^{6} x^{5} + a^{12} x^{4} + a^{3} x^{3} + 
        a^{4} x^{2} + a^{10} x + a^{12}
    )$,\\
   
$f_6(x,y)=(   a^{6} x^{11} + a^{14} x^{10} + x^{9} + x^{8} + a^{4} x^{7} + a^{2} x^{6} + a^{7} x^{5} + a^{13} x^{4} + a^{4} x^{3} + a^{12} x^{2} + 
        a^{5} x + a^{7 },\;
        a^{9} x^{14} + a^{6} x^{13} + a^{4} x^{12} + 
        a^{4} x^{11} + a^{12} x^{10} + a^{10} x^{9} + a^{7} x^{8} + a^{2} x^{7} + a^{11} x^{5} + a^{14} x^{4} + a^{4} x^{2} + a^{11} x + a
    )$,\\   
      
$f_7(x,y)=(  a^{8} x^{11} + a^{11} x^{10} + a^{12} x^{9} + a^{2} x^{8} + a^{14} x^{7} + a^{10} x^{6} + a^{4} x^{5} + a^{7} x^{4} + a^{2} x^{3} + 
        a^{13} x + a^{12 } ,\;
        a^{8} x^{14} + a^{10} x^{13} + a^{14} x^{12} + 
        a^{7} x^{11} + a^{5} x^{10} + a^{13} x^{9} + a^{13} x^{8} + a^{12} x^{7} + a^{7} x^{6} + a^{8} x^{5} + a^{12} x^{4} + x^{3} + 
        a^{12} x^{2} + a^{6} x + a^{10}
    )$,\\  
     
$f_8(x,y)=(  a^{13} x^{11} + a x^{10} + a^{11} x^{9} + a x^{8} + a^{9} x^{7} + a^{11} x^{6} + a^{10} x^{5} + a^{9} x^{4} + x^{3} + a^{4} x^{2} + 
        a^{6} x + 1  ,\;
        a^{2} x^{14} + a^{2} x^{13} + a^{11} x^{12} + 
        a^{5} x^{11} + a^{7} x^{10} + a^{2} x^{9} + a^{4} x^{8} + a^{11} x^{7} + a^{14} x^{6} + a^{13} x^{5} + a^{8} x^{4} + a^{4} x^{3} + 
        a^{6} x^{2} + x + a^{12}
    )$, \\ 
     
$f_9(x,y)=(  a^{10} x^{11} + a^{6} x^{10} + a^{12} x^{9} + a^{12} x^{8} + a^{7} x^{7} + a^{3} x^{6} + a^{12} x^{5} + a^{2} x^{4} + a^{6} x^{3} + 
        a^{12} x^{2} + a^{6} x + 1  ,\;
        a^{8} x^{14} + a^{11} x^{13} + 
        a x^{12} + a^{8} x^{11} + a^{5} x^{10} + a^{8} x^{9} + a^{3} x^{8} + a^{14} x^{7} + a^{9} x^{6} + a^{13} x^{5} + a^{11} x^{4} + 
        a^{3} x^{3} + a^{7} x^{2} + x + a^{2}
    )$,\\   
      
$f_{10}(x,y)=(  a^{14} x^{11} + a x^{9} + a^{12} x^{8} + a^{3} x^{7} + x^{6} + a^{7} x^{5} + a^{11} x^{4} + a^{14} x^{3} + a^{8} x^{2} + a x + 
        a^{7}  ,\;
        a^{3} x^{14} + a^{8} x^{13} + a^{7} x^{11} + a^{14} x^{10} + 
        a^{13} x^{9} + a^{9} x^{8} + a^{6} x^{7} + a^{9} x^{6} + a^{8} x^{5} + a^{12} x^{2} + a^{13} x + a^{14}
    )$,\\
      
$f_{11}(x,y)=(  a^{9} x^{11} + a^{5} x^{10} + a^{5} x^{9} + a x^{8} + a^{7} x^{7} + a^{5} x^{6} + a^{2} x^{5} + a^{4} x^{4} + a^{3} x^{3} + a^{2} x^{2} + 
        a x + a  ,\;
        a^{10} x^{14} + a^{4} x^{12} + a^{5} x^{11} + 
        a^{14} x^{10} + a^{5} x^{9} + a^{9} x^{7} + a^{7} x^{6} + a^{4} x^{5} + a^{14} x^{4} + a^{10} x^{3} + a^{9} x^{2} + a^{5} x + 
        a^{3}
             )$,\\
     
$f_{12}(x,y)=(  a^{11} x^{11} + a^{9} x^{10} + a^{10} x^{9} + a^{7} x^{8} + a^{10} x^{6} + a^{2} x^{5} + a^{13} x^{4} + a^{7} x^{3} + a^{2} x^{2} + 
        a^{11} x + a^{3 } ,\;
        a^{9} x^{14} + a^{11} x^{13} + x^{12} + 
        a^{11} x^{11} + x^{10} + a^{9} x^{9} + a^{8} x^{8} + a^{11} x^{7} + a^{8} x^{6} + a^{12} x^{5} + a^{6} x^{4} + a^{2} x^{3} + a^{4} x^{2
       } + a^{14} x + a^{13}
     )$, \\  
      
$f_{13}(x,y)=( a^{9} x^{10} + a x^{9} + a^{3} x^{8} + a^{10} x^{7} + a^{7} x^{6} + a^{12} x^{4} + a^{3} x^{3} + a^{12} x^{2} + a^{11} x + a^{8}   ,\;
a x^{14} + a^{12} x^{13} + a^{2} x^{12} + a^{3} x^{11} + x^{10} + 
        a^{12} x^{9} + a^{3} x^{8} + a^{14} x^{7} + x^{6} + a^{14} x^{5} + a^{10} x^{4} + a^{5} x^{3} + a^{5} x^{2} + a^{6} x + 1
     )$,\\

$f_{14}(x,y)=( a^{14} x^{11} + a^{4} x^{10} + a^{10} x^{9} + a^{6} x^{8} + a^{7} x^{7} + a^{14} x^{6} + a^{12} x^{5} + a x^{4} + a^{5} x^{3} + a^{4} x +
        a^{12 }  ,\;
        a^{14} x^{14} + a^{14} x^{13} + a^{3} x^{12} + a^{11} x^{11} + a^{12} x^{10} + a x^{9
       } + a^{5} x^{8} + a^{2} x^{7} + a^{6} x^{6} + a^{14} x^{5} + a^{6} x^{3} + a^{11} x^{2} + a^{4} x + a^{9}
    )$,\\
      
$f_{15}(x,y)=( a^{8} x^{11} + a^{8} x^{10} + a^{2} x^{9} + a^{13} x^{8} + a^{4} x^{7} + a x^{6} + a^{7} x^{4} + a^{3} x^{3} + a^{9} x^{2} + a^{5} x + 
        a^{5 }  ,\;
        a^{10} x^{14} + a^{6} x^{13} + x^{12} + a^{11} x^{11} + a^{4} x^{10} + a^{10} x^{9} + 
        a^{14} x^{8} + a^{13} x^{7} + a^{11} x^{6} + a x^{5} + a^{12} x^{4} + x^{2} + a^{5}
     )$,\\
     
$f_{16}(x,y)=(  a^{13} x^{11} + a^{8} x^{10} + a^{6} x^{9} + x^{8} + a^{12} x^{7} + a^{9} x^{6} + x^{5} + a^{10} x^{4} + a^{14} x^{3} + a^{4} x^{2} + 
        a^{8} x + a^{2 } ,\;
        a^{7} x^{14} + a^{11} x^{13} + x^{12} + a^{5} x^{11} + a^{10} x^{9} + 
        a^{9} x^{8} + a^{6} x^{7} + a^{9} x^{6} + a^{14} x^{5} + a^{13} x^{4} + a^{11} x^{3} + a^{6} x^{2} + a^{9} x + a^{11}
     )$, \\  
      
$f_{17}(x,y)=( a^{8} x^{11} + a^{10} x^{10} + a^{2} x^{9} + a x^{8} + a^{13} x^{6} + a x^{5} + a^{9} x^{4} + a^{3} x^{2} + a^{4} x + a^{10 }  ,\;
a^{9} x^{14} + a^{13} x^{12} + a^{7} x^{11} + a^{9} x^{10} + a^{12} x^{9} + a^{2} x^{8} + a^{3} x^{7} + 
        a^{13} x^{6} + a^{12} x^{5} + a^{5} x^{4} + a^{4} x^{3} + a^{8} x^{2} + a^{10} x + a^{9}
     )$,\\          
     
     $f_{18}(x,y)=(  x^{10} + a^{5} x^{9} + x^{8} + a^{5} x^{7} + a^{10} x^{6} + a^{10} x^{5} + a^{5} x^{4} + a^{4} x^{3} + a^{5} x^{2} + a^{4 } ,\;
     a^{3} x^{14} + a^{13} x^{13} + a^{5} x^{12} + a^{8} x^{11} + a^{12} x^{10} + a^{4} x^{9} + a^{10} x^{8} + 
        a^{11} x^{7} + a^{2} x^{6} + a^{12} x^{5} + a^{11} x^{4} + a^{2} x^{3} + a^{3} x^{2} + a^{12} x + a^{8}
    )$,\\
      
$f_{19}(x,y)=( a^{8} x^{11} + a^{13} x^{10} + a^{6} x^{9} + a^{6} x^{8} + a^{4} x^{6} + a^{10} x^{5} + a^{4} x^{3} + a^{3} x + 1   ,\;
a^{14} x^{14} + a x^{13} + a^{14} x^{12} + a^{5} x^{11} + a x^{10} + a^{2} x^{9} + a^{7} x^{8} + a x^{7} + 
        a^{8} x^{6} + a^{9} x^{5} + a^{2} x^{4} + a^{9} x^{3} + a^{8} x^{2} + a^{5} x + a^{12}
     )$,\\
     
$f_{20}(x,y)=( a x^{11} + a^{6} x^{10} + x^{9} + a^{5} x^{8} + a^{5} x^{7} + a^{9} x^{6} + a^{12} x^{5} + a^{4} x^{4} + x^{3} + a^{13} x^{2} + a^{4} x +
        1   ,\;
        a x^{14} + x^{13} + a^{6} x^{11} + a^{5} x^{10} + a^{9} x^{9} + a^{10} x^{8} + a^{9} x^{7 
       } + a^{5} x^{6} + a^{8} x^{5} + a^{10} x^{4} + a^{11} x^{2} + a^{6} x + a^{3}
     )$, \\  
      
$f_{21}(x,y)=(a^{10} x^{11} + a^{14} x^{10} + a^{13} x^{8} + a^{2} x^{7} + a^{11} x^{6} + a^{7} x^{5} + a^{7} x^{3} + x^{2} + a^{4} x + a^{7 }   ,\;
a^{6} x^{14} + a^{4} x^{13} + a^{8} x^{12} + a x^{11} + a^{11} x^{10} + a^{4} x^{9} + a^{7} x^{8} + 
        a x^{7} + a^{9} x^{6} + a^{12} x^{5} + a^{6} x^{4} + a^{4} x^{3} + a^{2} x^{2} + a^{5} x + a^{7}
     )$,\\          
     
     $f_{22}(x,y)=( a^{10} x^{11} + x^{10} + a^{7} x^{8} + x^{7} + a^{13} x^{6} + a^{6} x^{5} + a^{2} x^{4} + a^{10} x^{3} + a^{2} x^{2} + a x + 1   ,\;
     a^{2} x^{14} + a^{2} x^{13} + a^{3} x^{11} + a^{11} x^{10} + a^{10} x^{8} + a^{11} x^{7} + a^{11} x^{6 
       } + a^{12} x^{5} + a x^{4} + a^{14} x^{3} + a^{2} x^{2} + a^{3} x
    )$,\\
      
$f_{23}(x,y)=( a^{7} x^{11} + a^{9} x^{10} + a^{7} x^{9} + a^{2} x^{8} + a^{8} x^{7} + a^{11} x^{6} + a^{14} x^{5} + a^{7} x^{4} + a^{10} x^{3} + 
        a^{8} x^{2} + a^{4} x   ,\;
        a^{10} x^{14} + a^{8} x^{13} + a^{5} x^{12} + x^{11} + a^{6} x^{10} + 
        a^{12} x^{9} + a^{11} x^{8} + a^{6} x^{7} + a^{2} x^{6} + a^{8} x^{5} + a^{3} x^{4} + a^{7} x^{3} + a^{12}
     )$,\\
     
$f_{24}(x,y)=( x^{11} + a^{5} x^{10} + a^{3} x^{9} + a^{10} x^{8} + a^{10} x^{7} + a^{4} x^{6} + a^{14} x^{5} + a^{8} x^{3} + a^{3} x^{2} + a^{6} x + 
        a^{8 }  ,\;
        a^{10} x^{13} + a^{13} x^{12} + a^{13} x^{11} + a x^{10} + a^{11} x^{8} + a^{2} x^{7} +
        a x^{6} + a^{2} x^{5} + a^{2} x^{4} + a^{8} x^{2} + a^{6}
     )$, \\  
      
$f_{25}(x,y)=( a^{4} x^{11} + a^{12} x^{10} + a^{4} x^{9} + a^{10} x^{8} + a^{12} x^{7} + a^{14} x^{6} + a^{7} x^{5} + a^{11} x^{4} + a^{13} x^{3} + 
        a^{13} x^{2} + a^{10} x + a^{4 }  ,\;
        x^{14} + a^{6} x^{13} + a^{3} x^{12} + a^{11} x^{11} + 
        a^{13} x^{10} + a^{11} x^{9} + a^{7} x^{8} + a^{2} x^{7} + a^{5} x^{6} + a^{14} x^{5} + a^{5} x^{4} + a^{10} x^{3} + a^{3} x^{2} + 
        a^{14} x + a^{5}
     )$.

 \vspace{1em}
 
 \subsubsection{The basis of $\mathcal{L}(D_1+D_2)$}
 
 As seen in Section \ref{baseLD1D2}, ${\mathcal B}_{D_1+D_2}=(f_1,\ldots,f_n,f_{n+1}, \ldots ,f_{2n+g-1})$ where,
for  $j\in \{1, \ldots , 2n-1\}$, the $f_i$ are defined above.  The basis is completed 
with $$f_{26}(x,y)=  \frac{g_{26}(x)y + h_{26}(x)}{r(x)} \textrm{ and }
			f_{27}(x,y)=  \frac{g_{27}(x)y + h_{27}(x)}{r(x)}$$
     where: \\
     
    $ g_{26}(x) =
     		a^{13} x^{25} + a^{14} x^{24} + a^{2} x^{23} + a^{8} x^{22} + a^{7} x^{21} 
     		+ a x^{19} + a^{9} x^{18} + a^{7} x^{17} + a^{6} x^{16} + a^{10} x^{15} 
     		+ a^{10} x^{14} + a^{2} x^{13} + a^{5} x^{12} + a^{4} x^{11} + a^{11} x^{9} 
     		+ a^{12} x^{8} + a^{7} x^{7} + a^{14} x^{6 }  + x^{5} + a^{8} x^{4} + a^{9} x^{3} + a^{7} x^{2 }$, \\
    
    $h_{26}(x)=
        a^{13} x^{28} + a^{11} x^{26} + x^{25} + a^{4} x^{24} + a^{5} x^{23} + a^{3} x^{22} + a^{3}x^{21} 
        + x^{20} + a^{12} x^{19} + a^{4} x^{18} + a^{13} x^{16} + a^{5} x^{15} + a^{11} x^{14} 
        + a^{6} x^{13} + a x^{12} + x^{11} + a^{8} x^{10} + a^{8} x^{9} + a^{12} x^{7} + a^{2} x^{6} 
        + a^{2} x^{5} + x^{4} + a^{8} x^{3} + a^{14} x^{2} + a^{10} x + a^{10}$,\\
      
	$g_{27}(x)= a^{4} x^{25} + a^{6} x^{24} + a^{6} x^{23} + a^{2} x^{22} + a^{3} x^{21} 
		+ a^{9} x^{20} + a^{2} x^{19} + a^{14} x^{18} + a^{9} x^{17} + a^{8} x^{16} 
		+ a^{13} x^{15} + a^{12} x^{14} + a^{8} x^{13} + x^{12} + a^{3} x^{10} + a^{4} x^{9} 
		+ a^{12} x^{8} + a^{6} x^{7} + a^{2} x^{6} + a^{14} x^{5} + a^{5} x^{4} + a^{4} x^{3} 
		+ a^{4} x^{2} + a^{5} x + a^{5 }$,\\
        
    $h_{27}(x) = a^{9} x^{28} + a^{5} x^{27} + a^{6} x^{26} + a^{4} x^{25} + x^{24 }
       + a^{7} x^{23} + a^{5} x^{22} + a^{6} x^{21} + x^{20} + a^{12} x^{19} + a^{7} x^{18} 
       + a^{11} x^{17} + a^{8} x^{16} + a^{7} x^{15} + a^{6} x^{14} + a^{4} x^{13} 
       + a^{7} x^{12} + a^{3} x^{11} + a^{9} x^{10} + x^{9} + a^{5} x^{8} + a^{9} x^{7} 
       + a^{9} x^{6} + a^{5} x^{5 } + x^{4} + a^{13} x^{2} + a^{3} x + a^{2}$,\\

	$r(x) = x^{28} + a^9x^{27} + a^6x^{26} + a^7x^{25} + a^{11}x^{24} + a^{12}x^{23} 
	+ a^{10}x^{22} + a^6x^{21} + a^7x^{20} + a^{10}x^{19} + a^{14}x^{18} + x^{17} + a^3x^{14} + a^9x^{13} 
	+ a^{10}x^{12} + a^7x^{11} + a^2x^{10} + a^{13}x^9 + a^{10}x^8 + a^{11}x^7 + a^6x^6 
	+ a^{10}x^5 + a^9x^4 + a^7x^3 + a^6x + a$.

\section{Multiplication in $\F_{4^{n}/\F_{4}}$}\label{effective2}

Set $q=4$ and $n=4$. With the algebraic function field $F/\F_q$ defined over $\F_4$
 associated to the hyperelliptic curve $X$ with plane model $y^2+y=x^5$, of genus two, we cannot multiply with places of degree one.
%
Indeed, the number of places of degree one is $N_1(F/\F_4)=5$. So, if we use only places of degree one
we get that $\dim  \Ima (T) \le N_1(F/\F_4)$ whereas $\dim \mathcal{L}(D_1+D_2) =  2n+g-1 = 10$, 
namely the evaluation map $T$ is not injective as required in Theroem~\ref{AlgoRandriam}.

We can ask if there exists another hyperelliptic curve of genus two that allows the multiplication with only 
places of degree one in $\F_{4^4}$. 
By definition, an hyperelliptic curve 
 is a covering of degree two of the projective line $\P_1(\F_q)$.
 It is clear that the maximal number $N_q(g)$ of rational points of a projective smooth absolutely irreducible 
 curve of genus $g=2$ over $\F_q$ is such that $N_q(2)\leq 2q+2$. 
 In the case of $q=4$ and $g=2$, this bound is least that the bound of Serre-Weil $N_q(2)\leq q+1+2m$ where $m=\lfloor 2{\sqrt q}\rfloor$. Hence, there does 
 not exist a maximal Serre-Weil curve (i.e. attaining the Serre-Weil bound) over $\F_4$ of genus two and $N_4(2)\leq 2q+2=q+1+2m-3=10$. In fact, Serre in \cite{Serr1} 
 proves that $N_4(2)=10$ and Shabat in  \cite{Shab} exhibits an optimal (i.e. attaining $N_q(g))$ curve of genus two over $\F_4$, namely the curve of equation $y^2+y=\frac{x}{x^3+x+1}$.
 Clearly, the sufficient Condition (\ref{conditionnbrepoints}) fails since
  $N_1(F/\F_4)>2n+2g-2$  implies that $n<4$.
 Nevertheless,  it is still possible to multiply according to Theorem~\ref{AlgoRandriam} 
 with  the ten places of degree one of the algebraic function field $F/\F_4$ associated to the curve $y^2+y=\frac{x}{x^3+x+1}$,  by choosing suitable degree $n$ place $Q$ and divisors 
 $D_1$ and $D_2$ of $F/\F_4$. For instance, using the description given by Magma, we can choose
 the $n$-degree place
 $$Q = (x^4 + a^2x^2 + a^2x + a, (x^3 + x + 1)y + x^2 + a^2x + a)$$ 
 that lies over the place $(\mathcal{Q}(x))$ of the rational field $\F_4(x)$ defined by
$\mathcal{Q}(x)=  x^4 + a^2x^2 + a^2x + a$, 
and the divisors  
$$D_1 = (x^5 + x^4 + x^3 + x^2 + x + a^2, (x^3 + x + 1)y + a^2x^3 + a^2x^2 + a^2x + a^2)$$ and 
$$D_2 = (x^5 + x + a, (x^3 + x + 1)y + a^2x^4 + a^2x^3 + a^2x + a^2)$$ 
of degree $n+g-1 = 5$ that respectively lie over the places $(\mathcal{D}_1(x))$ and $(\mathcal{D}_2(x))$
of the rational field $\F_4(x)$ defined by
$\mathcal{D}_1(x) = x^5 + x^4 + x^3 + x^2 + x + a^2$ and  $\mathcal{D}_2(x) =  x^5 + x + a$. 
 
Now, suppose that $n=5$.
In this case,  we have $\dim \mathcal{L}(D_1+D_2) =  2n+g-1 = 11$ and still
$\dim  \Ima (T) \le N_1(F/\F_4) = 10$. This means that $T$ can not be injective.
Since the curve of equation $y^2+y=\frac{x}{x^3+x+1}$ is optimal, 
it is impossible to multiply with only places of degree one in extensions of $\F_4$ of degree $n\geq 5$. 
Consequently, the algorithm has to be modified in order to use 
places of degree two. It is clear that we need to use as much as possible places of degree one to minimize bilinear complexity. Therefore, our algorithm will still  use 
the curve of equation  $y^2+y=\frac{x}{x^3+x+1}$.  
%
 This curve has $N_1(F/\F_4)=10$ rational places and $N_2(F/\F_4)=4$ places of degree two. 
 We represent $\F_{4}$ as the field $\F_2(a)=\F_2[X]/(P_1(X))$ where $P_1(X)$ is the primitive irreducible polynomial $P_1(X)= X^2+X+1$ and 
 $a$ denotes a primitive root of $P_1(X)= X^2+X+1$. As we also use evaluations over places of degree two, we need to define the finite field 
 $\F_{16}=\F_{4^2}$ as $\F_4$-vector space. So, we represent $\F_{4^2}$ as the field $\F_4(b)=\F_4[X]/(P_2(X))$ where $P_2(X)$ is the primitive irreducible polynomial $P_2(X)= X^2+X+a$ and 
 $b$ denotes a primitive root of $P_2(X)= X^2+X+a$. 
 

\vspace{.5em}
 
 For the description of the places of degree one, we use the description given by Magma:
 
  $$
 \begin{array}{ll}
 P_{\infty,1}=(\frac{1}{x}, \frac{x^3 + x + 1}{x^3}y + \frac{1}{x}) & 
 P_{\infty,2}=(\frac{1}{x}, \frac{x^3 + x + 1}{x^3}y + \frac{x + 1}{x}) \\
 P_{3}=(x, (x^3 + x + 1) y) & P_{4}=(x, (x^3 + x + 1) y + x + 1)\\
 P_{5}=(x + a, (x^3 + x + 1) y + 1) & P_{6}=(x+ a, (x^3 + x + 1) y + x + 1)\\
 P_{7}=(x + a^2, (x^3 + x + 1) y + 1) & P_{8}=(x + a^2, (x^3 + x + 1) y + x + 1)\\
 P_{9}=(x + 1, (x^3 + x + 1) y + a) & P_{10}=(x + 1, (x^3 + x +1) y + a^2).
\end{array}
$$ 
  
 Note that the first two infinite places lie above the infinite place $1/x$ of the rational function field $\F_q(x)$. 
 Moreover, if $(x:y:z)$ denotes the projective coordinates of rational points of the curve $X$, then these infinite places correspond to 
 the two points to the infinity $P_{\infty,1}=(0: 1 : 0)$ and $P_{\infty,1}=(1: 0 : 0)$ of $X$.
 
 \vspace{.5em}
 
 For the description of the places of degree two, we use the description given by Magma:
 
 
   $$
 \begin{array}{ll}
	Q_1=(x^2 + a x + a, (x^3 + x + 1) y + a^2 x + a^2) & Q_2=(x^2 + a x + a, (x^3 + x +1) y + a^2 x + 1)\\
	Q_3=(x^2 + a^2 x + a^2, (x^3 + x + 1) y + a x + a) & Q_4=(x^2 + a^2 x + a^2, (x^3 + x + 1) y + a x + 1).
\end{array}
$$

%

\vspace{.5em}


Again, we proceed as in Subsection~\ref{SS-constructionOfDivisors}.
We choose an  irreducible polynomial $\mathcal{Q}(x)$ of degree $n = 5$
and  two irreducible polynomials $\mathcal{D}_1(x)$ and $\mathcal{D}_2(x)$ of degree $n+g-1 = 6$.
For instance,
$$\begin{array}{l}
 \mathcal{Q}(x)= x^5 + ax^4 + x^3 + a^2x^2 + ax + 1,\\
  \mathcal{D}_1(x) = x^6 + ax^4 + ax^2 + x + a^2, \\
   \mathcal{D}_2(x) = x^6 + ax^3 + a^2x^2 + a^2.
 \end{array}$$
The  degree $n$ place $(\mathcal{Q}(x))$ and 
the two degree $n+g-1$ places $(\mathcal{D}_1(x))$ and $(\mathcal{D}_2(x))$ of $\F_{q}(x)/\F_q$
totally split in $F/\F_q$.
Then we choose suitable places $Q$, $D_1$ and $D_2$ of $F/\F_q$  lying over the places 
$(\mathcal{Q}(x))$, $(\mathcal{D}_1(x))$ and $(\mathcal{D}_2(x))$ respectively,
that is, such that $D_1-Q$ and $D_2-Q$ are non-special divisors of degree $g-1$. 
For instance, using the description of Magma,
$$\begin{array}{l}
Q = (x^5 + ax^4 + x^3 + a^2x^2 + ax + 1, (x^3 + x + 1)y + ax^4 + a^2x^3 + ax^2 + a^2x + 1),\\
D_1 = (x^6 + ax^4 + ax^2 + x + a^2, (x^3 + x + 1)y + x^5 + a^2),\\
D_2 = (x^6 + ax^3 + a^2x^2 + a^2, (x^3 + x + 1)y + a^2x^5 + ax^2 + a^2).
\end{array}$$

  
As in Section \ref{baseLD1},  
we choose as basis of the Riemann-Roch space $\mathcal{L}(D_i)$ the basis ${\mathcal B}_{D_i }$
  such that $E_i({\mathcal B}_{D_i})={\mathcal B}_Q$ 
  is a basis of $F_Q$. 
We write ${\mathcal B}_{D_1}=(f_1,...,f_n)$ and  ${\mathcal B}_{D_2}=(f_1,f_{n+1}...,f_{2n-1})$. 
  For $j\in \{2, \ldots , n\}$ and $k\in \{n+1, \ldots , 2n-1\}$, any element $f_j$ of ${\mathcal B}_{D_1}$  
  and $f_k$ of ${\mathcal B}_{D_2}$ are respectively of the form:
	$$f_j(x,y)=\frac{f_{j1}(x)y+f_{j2}(x)}{\mathcal D_1(x)} 
	\mathrm{~~and~~ } f_k(x,y)=\frac{f_{k1}(x)y+f_{k2}(x)}{\mathcal D_2 (x)},$$ 
  where $f_{j1}, f_{j2}, f_{k1}, f_{k2}\in \F_{4}[x]$. 
  To simplify, we set $f_j(x,y)=(f_{j1}(x), f_{j2}(x))$ for all $j\in \{2,2n-1\}$. 
 We have:
  
  \vspace{1em}

$f_1(x,y)=1$,\\

$f_2(x,y)=(a^2x^5 + a^2x^4 + ax^3 + ax + 1,  a^2x^6 + a^2x^5 + a^2x^4 + a^2x^3 + x^2 + ax + a)$,\\

$f_3(x,y)=(a^2x^6 + x^5 + x^3 + ax^2 + a^2,  x^6 + x^4 + ax^3 + ax^2 + x + 1),$\\

$f_4(x,y)=(ax^6 + a^2x^5 + ax^4 + x^3 + a^2x^2, ax^6 + ax^5 + a^2x^3 + ax^2 + x + a)$,\\

$f_5(x,y)=(x^6 + a^2x^5 + a^2x^3 + ax^2 + 1,  a^2x^5 + x^4 + a^2x^2 + a^2x + a^2)$,\\

$f_6(x,y)=(a^2x^6 + ax^4 + x^2 + ax + a^2,  ax^6 + a^2x^4 + a^2x^3 + x)$,\\

$f_7(x,y)=(x^6 + x^2 + 1, ax^4 + ax^3 + a^2x + a^2)$,\\

$f_8(x,y)=(a^2x^5 + x^4 + a^2x^3 + ax^2 + x, x^6 + a^2x^5 + a^2x^4 + a^2x^3 + x^2 + a^2x + 1)$,\\

$f_9(x,y)=(x^5 + x^4 + ax^3 + ax + a^2, x^6 + a^2x^5+ ax^4 + x^3 + a^2x^2 + a^2x)$,\\

The basis  ${\mathcal B}_{D_1+D_2}=(f_1,\ldots ,f_{2n+g-1})$ consists of
the $f_i$'s that are defined above for  $j\in \{1, \ldots , 2n-1\}$ and the $g=2$ components 
$$f_{10}(x,y)=  \frac{g_{10}(x)y + h_{10}(x)}{r(x)} \textrm{ and }
			f_{11}(x,y)=  \frac{g_{11}(x)y + h_{11}(x)}{r(x)}$$
     where: \\

$g_{10}(x)=ax^{12} + x^{11} + a^2x^9 + a^2x^8 + ax^7 + ax^6 + a^2x^5 + x^3 + x^2 +
   ax$,\\
   
   $ h_{10}(x)= x^{12} + a^2x^9 + ax^8 + a^2x^7 + x^6 + ax^4 + ax^3 + x^2 + a^2x + 1)$,\\

$f_{11}(x)=a^2x^{12} + a^2x^{11} + a^2x^{10} + x^7 + ax^5 + ax^3 + ax^2 + a^2x + 1$,\\

$h_{11}(x) = x^{10} + x^8 + ax^5 + a^2x^2 + 1$,\\
   
$r(x) = x^{12} + ax^{10} + ax^9 + x^8 + ax^7 + x^6 + a^2x^5 + ax^4 + ax^3 + a^2x^2 + a^2x + a$.

\section{Multiplication in $\F_{2^{n}/\F_{2}}$}\label{effective3}

Set $q=2$ and $g=2$. In this case, it is known that $N_2(g)=6$ by the upper bound of Ihara \cite{Ihara} 
and the lower bound of Serre \cite{Serr1} and it is not sufficient to multiply in the extensions of $\F_{2}$ of degree $n\geq 3$. 
Hence, we need to use places of higher degree. Note that as by the above section, 
$N_4(g)=10=N_1(F/\F_2)+2N_2(F/\F_2)$, we can only multiply in extensions of degree $n\leq 4 $ if we only use places of degree one and two 
(for any curve of genus two!). By consequence,  we set $n=5$ and we consider the curve used in Section \ref{effective2} namely 
the algebraic function field $F/\F_2$ associated to the hyperelliptic curve $X$ with plane model $y^2+y=\frac{x}{x^3+x+1}$, of genus two. 
This curve has $N_1(F/\F_2)=4$ rational places, $N_2(F/\F_2)=3$ places of degree two and $N_4(F/\F_2)=2$ places of degree four. 
 For the description of the places, we use the description given by Magma. The places of degree one are :

  $$
 \begin{array}{ll}
P_{\infty,1} = (\frac{1}{x}, \frac{x^3 + x + 1}{x^3}y + \frac{1}{x}) 
& P_{\infty,2} = (\frac{1}{x}, \frac{x^3 + x + 1}{x^3}y + \frac{x + 1}{x})\\
P_3 =  (x, (x^3 + x + 1)y)
& P_4 = (x, (x^3 + x + 1)y + x + 1).
\end{array}
$$
The places of degree two are:
$$
 \begin{array}{lll}
Q_1 = (x + 1) & Q_2= (x^2 + x + 1, (x^3 + x + 1)y + 1) & Q_3= (x^2 + x + 1, (x^3 + x + 1)y + x  + 1).
\end{array}
$$ 
And the places of degree four are:    
$$
 \begin{array}{ll}
 R_1 = (x^4 + x^3 + 1, (x^3 + x + 1)y + x^2 + x + 1) & R_2 =  (x^4 + x^3 + 1, (x^3 + x + 1)y + x^3 + x^2).
 \end{array}
 $$

According to our method,
we choose the  irreducible polynomial $\mathcal{Q}(x) = x^5 + x^3 + 1$ of degree $n = 5$
and the  two irreducible polynomials $\mathcal{D}_1(x) = x^6 + x^5 + x^4 + x + 1$ 
and $\mathcal{D}_2(x) = x^6 + x^5 + x^2 + x + 1$ of degree $n+g-1 = 6$.
The  degree $n$ place $(\mathcal{Q}(x))$ and 
the two degree $n+g-1$ places $(\mathcal{D}_1(x))$ and $(\mathcal{D}_2(x))$ of $\F_{q}(x)/\F_q$
totally split in $F/\F_q$.
Then we choose  places $Q$, $D_1$ and $D_2$ of $F/\F_q$  lying over these three places 
and such that $D_1-Q$ and $D_2-Q$ are non-special divisors of degree $g-1$. 
Using the description of Magma,
$$\begin{array}{l}
Q = (x^5 + x^3 + 1, (x^3 + x + 1)y + x^4 + x + 1)\\
D_1 = (x^6 + x^5 + x^4 + x + 1, (x^3 + x + 1)y + x^5 + x^3 + 1)\\
D_2 = (x^6 + x^5 + x^2 + x + 1, (x^3 + x + 1)y + x^5 + x^4 + x).
\end{array}$$

  
As in Section \ref{baseLD1},  
we choose as basis of the Riemann-Roch space $\mathcal{L}(D_i)$ the basis ${\mathcal B}_{D_i }$
  such that $E_i({\mathcal B}_{D_i})={\mathcal B}_Q$ 
  is a basis of $F_Q$. 
We write ${\mathcal B}_{D_1}=(f_1,...,f_n)$ and  ${\mathcal B}_{D_2}=(f_1,f_{n+1}...,f_{2n-1})$. 
  For $j\in \{2, \ldots , n\}$ and $k\in \{n+1, \ldots , 2n-1\}$, any element $f_j$ of ${\mathcal B}_{D_1}$  
  and $f_k$ of ${\mathcal B}_{D_2}$ are respectively of the form:
	$$f_j(x,y)=\frac{f_{j1}(x)y+f_{j2}(x)}{\mathcal D_1(x)} 
	\mathrm{~~and~~ } f_k(x,y)=\frac{f_{k1}(x)y+f_{k2}(x)}{\mathcal D_2 (x)},$$ 
  where $f_{j1}, f_{j2}, f_{k1}, f_{k2}\in \F_{4}[x]$. 
  To simplify, we set $f_j(x,y)=(f_{j1}(x), f_{j2}(x))$ for all $j\in \{2,2n-1\}$. 
 We have:
  
  \vspace{1em}

$f_1(x,y)=1$,\\

$f_2(x,y)=(x^3 + x + 1,  x^6 + x^4 + 1)$,\\

$f_3(x,y)=(x^4 + x^3 + x^2 + 1,  x^4 + x^2 + 1),$\\

$f_4(x,y)=(x^5 + x^4 + x^3 + x, x^5 + x^3 + x)$,\\

$f_5(x,y)=(x^6 + x^5 + x^4 + x^2,  x^6 + x^4 + x^2)$,\\

$f_6(x,y)=(x^6 + x^5 + x^4 + x^3 + x^2 + x + 1,  x^5)$,\\

$f_7(x,y)=(x^6 + x^5 + x^3 + 1, x^5 + x^4 + x^2 + 1)$,\\

$f_8(x,y)=(x^6 + x^2 + 1, x^4 + x^3 + x^2 + x + 1)$,\\

$f_9(x,y)=(x^4 + x^3 + x^2 + 1, x^6 + x^3 + x + 1)$,\\

The basis is completed with 

$$f_{10}(x,y)=\frac{(x^{12} + x^8 + x^6 + x^5 + x^2 + x + 1)y
	+ (x^{11} + x^8 + x^5 + x^4 + x^3 + x^2 + x)}{r(x)}$$

$$f_{11}(x,y)=\frac{(x^{12} + x^{11} + x^9 + x^8 + x^6 + x)y + (x^{12} + x^{10} + x^8 + x^7 + x^6 + x^4 + x^3 + x^2)}{r(x)}$$

with $r(x) = x^{12} + x^9 + x^8 + x^7 + x^6 + x^5 + x^4 + x^3 + 1$.
\appendix

\section{Magma implementation of the multiplication algorithms in the finite fields }
\subsection{$\F_{16^{13}}$ over $\F_{16}$}
{\small
\begin{verbatim}
// Asymmetric version of Chudnovsky multiplication algorithm in GF16^13 
// using only places of degree 1

n:=13; g:=2; q:=16;
F16<a>:=GF(16);

// %%%%%%%%%%%%%%%%%%%%%%%%%%%
// ALGEBRAIC FUNCTION FIELD WITH CURVE y^2 + y + x^5 OF GENUS 2
// %%%%%%%%%%%%%%%%%%%%%%%%%%%

Kx<x> := FunctionField(F16); 
Kxy<y> := PolynomialRing(Kx);
f:=y^2 + y + x^5;
F<c> := FunctionField(f);

//------ find places of higher degree
LP1:=Places(F,1);    // 33 degree 1 places     

// %%%%%%%%%%%%%%%%%%%%%%%%%%%
// CHOOSE GOOD PLACE Q AND GOOD DIVISORS D1, D2
// %%%%%%%%%%%%%%%%%%%%%%%%%%%

//------ q<x>:=RandomIrreduciblePolynomial(F16,n);
q := x^13 + a^6 *x^12 + a^5*x^11 + a^11*x^10 + x^9 + a^12 *x^8 
 + a^7*x^7 + a^7*x^5 + a^2*x^4 + a^11*x^3 + a^8*x^2 + a^6*x+a^14;
Q := Decomposition(F,Zeros(Kx!q)[1])[1]; 
K<b>:=ResidueClassField(Q);
"degree of Q is ", Degree(Q); // n

//------ D1 := RandomIrreduciblePolynomial(F16,n+g-1);
D1 := x^14 + a^9*x^13 + a^6*x^12 + a^7*x^11 + a^11*x^10 + a^12*x^9 + a^10*x^8 + a^6*x^7 + a^7*x^6 
+ a^10*x^5 + a^14*x^4 + x^3 + x^2 + a^3*x + a;
D1:=Decomposition(F,Zeros(Kx!D1)[1])[1]; 
D1:=1*D1;

//------ D2:=RandomIrreduciblePolynomial(F16,n+g-1);
D2 := x^14 + x^2 + a*x + 1;
D2:=Decomposition(F,Zeros(Kx!D2)[1])[1]; 
D2:=1*D2;

//------ Check D1 and D2 are suitable
"D1-Q is special ? ",IsSpecial(D1-Q); // false
"dim L(D1) is ", Dimension(D1); // n
"D2-Q is special ? ", IsSpecial(D2-Q);  // false
"dim L(D2) is ", Dimension(D2); // n
"Is D1 equivalent to D2 ? ", D1 eq D2; // false
"dim L(D1+D2) is ", Dimension(D1+D2); // 27

// %%%%%%%%%%%%%%%%%%%%%%%%%%%
// CONSTRUCTION OF THE RIEMANN-ROCH SPACES
// %%%%%%%%%%%%%%%%%%%%%%%%%%%

LD1, h1 :=RiemannRochSpace(D1);
BD1 := h1(Basis(LD1)); 
LD2, h2 :=RiemannRochSpace(D2);
BD2 := h2(Basis(LD2));
LD1D2, h := RiemannRochSpace(D1+D2);

// %%%%%%%%%%%%%%%%%%%%%%%%%%%
// SET GOOD BASES
// %%%%%%%%%%%%%%%%%%%%%%%%%%%

//------ Construction of E1=Evalf(Q) and set a good basis for L(D1)
L:=[]; for i in [1..n] do L:=Append(L,ElementToSequence(Evaluate(BD1[i],Q))); end for;
E1:=Transpose(Matrix(L));
BasisLD1 := Matrix(F,1,n,BD1)*Matrix(F,E1^-1);
//BasisLD1;

//------ Construction of E2=Evalf(Q) and set a good basis for L(D2)
L:=[]; for i in [1..n] do L:=Append(L,ElementToSequence(Evaluate(BD2[i],Q))); end for;
E2:=Transpose(Matrix(L));
BasisLD2 := Matrix(F,1,n,BD2)*Matrix(F,E2^-1);
//BasisLD2;

//------ Merge the two previous bases to a basis of L(D1+D2) 
L1 := ElementToSequence(BasisLD1); 
L2 := ElementToSequence(BasisLD2);
// Concatenate L1 to L2 except the first component of L2
LL := [LD1D2!L1[i] : i in [1..n]] cat [LD1D2!L2[i] : i in [2..n]];
BasisLD1D2 := h(ExtendBasis(LL,LD1D2));

//------ In addition, we require that the two last elements of the basis are evaluated to 0
X := Transpose(Matrix(F,[ElementToSequence(Evaluate(BasisLD1D2[i],Q)): i in [2*n+g-2..2*n+g-1]]));
TT := Matrix(F,1,n,L1)*X;
BasisLD1D2[2*n+g-2] := BasisLD1D2[2*n+g-2] - ElementToSequence(TT)[1];
BasisLD1D2[2*n+g-1] := BasisLD1D2[2*n+g-1] - ElementToSequence(TT)[2];
BLD1D2 := ExtendBasis([LD1D2!BasisLD1D2[i] : i in [1..n]], LD1D2);
"Basis of BLD1D2 : "; 
for i in [1..2*n+g-1] do  printf "f%o=",i; BasisLD1D2[i]; end for;
//for i in [1..2*n+g-1] do  Evaluate(BasisLD1D2[i],Q); end for;

// %%%%%%%%%%%%%%%%%%%%%%%%%%%
// CONSTRUCTION OF T AND T^-1 USING PLACES OF DEGREE ONE
// %%%%%%%%%%%%%%%%%%%%%%%%%%%

// the rows of T are the evaluation on the degree 1 places over F16
ST:=[]; 
for j:=1 to 2*n+g-1 do 
    for i:=1 to 2*n+g-1 do 
        ST:=Append(ST, Evaluate(BasisLD1D2[i], LP1[j]));
    end for; 
end for; 

T := Matrix(2*n+g-1,2*n+g-1, ST);
"Rank of T : ", Rank(T);
TI := T^-1;

// %%%%%%%%%%%%%%%%%%%%%%%%%%%
// ALGORITHM FOR THE MULTIPLICATION
// %%%%%%%%%%%%%%%%%%%%%%%%%%%

// ============= FUNCTION MULT =============
// @parameter : VarX, VarY are the coordinates  in a canonical basis
//			of the elements of (F16)^n to multipliate
// @return : the result of VarX * VarY in the canonical basis of (F16)^n

mult := function(varX, varY) 

	//------ injection of the coordinates of X into L(D1+D2) using basis BLD1D2
	fx := VerticalJoin(varX,ZeroMatrix(F16,n+g-1,1));

	//------ injection of the coordinates of Y into L(D1+D2) using basis BLD1D2
	Y1 := [varY[1,1]] cat [0 : i in [2..n]] cat [varY[i,1] : i in [2..n]] cat [0,0];
	fy := Matrix(F16,2*n+g-1,1,Y1);

	//------ Hadamard product  u = T(fx)*T(fy)
	u:=ZeroMatrix(F16,2*n+g-1,1);
	// the products are done in F16 for all coordinates
	TFX:=T*fx; 
	TFY:=T*fy;
	for i:=1 to 2*n+g-1 do u[i,1]:=TFX[i][1]*TFY[i][1]; end for;

	//------ E_Q(TI(u)) : T^-1 then evaluation in Q
	uu:=Matrix(F,TI*u);
	result := Evaluate(Matrix(1,2*n+g-1,BasisLD1D2)*uu,Q);
	return result;
end function; 
// =============  END FUNCTION MULT =============

// %%%%%%%%%%%%%%%%%%%%%%%%%%%
// EXAMPLES
//%%%%%%%%%%%%%%%%%%%%%%%%%%%

// Example 1 : (a+b)*(1+a*b+a*b^2) = a*b^3 + a^5*b^2 + a^8*b + a
X1 := Matrix(F16,n,1,ElementToSequence(a+b)); 
Y1 := Matrix(F16,n,1,ElementToSequence(1+a*b+a*b^2));
mult(X1,Y1);
 
// Example 2 : b^5 * (1 + a^2*b^3 + b^4) = b^9 + a^2*b^8 + b^5
X2 := Matrix(F16,n,1,ElementToSequence(b^5)); 
Y2 := Matrix(F16,n,1,ElementToSequence(1 + a^2*b^3 + b^4)); 
mult(X2,Y2);

// Example 3 : (a*b + b^2 + a*b^4) * (a*b + b^2 + a*b^4) = a^2*b^8 + b^4 + a^2*b^2
X3 := Matrix(F16,n,1,ElementToSequence(a*b + b^2 + a*b^4)); 
Y3 := Matrix(F16,n,1,ElementToSequence(a*b + b^2 + a*b^4)); 
mult(X3,Y3);

// Example 4 : (a*b + b^2 + a*b^4 + b^7 + a*b^12) * (a*b + b^2 + a*b^4) = 
// a*b^12 + a^8*b^11 + b^10 + a^2*b^9 + a^8*b^8 + a^3*b^7 + a^9*b^6 + a^11*b^5 +
//  a^8*b^4 + a^5*b^3 + a^3*b^2 + a^8*b + a
X4 := Matrix(F16,n,1,ElementToSequence(a*b + b^2 + a*b^4 + b^7 + a*b^12)); 
Y4 := Matrix(F16,n,1,ElementToSequence(a*b + b^2 + a*b^4)); 
mult(X4,Y4);
\end{verbatim}
}

 \subsection{$\F_{4^{5}}$ over $\F_{4}$}
{\small
\begin{verbatim}
// Asymmetric version of Chudnovsky multiplication algorithm in GF4^5 % CHOISIR UNE NOTATION ENTRE GF (Galois Field) ou F pour \UTF{00EA}tre homog\`ene.
// In this case, the curve is defined over F4 instead of F16 
// and, for n:=5, we must use places of degree 2.
// Note that in GF4^4 (i.e. when n:= 4) it suffices to use degree 1 places. % IDEM

n:=5; g:=2; q:=4;
F4<a>:=GF(4);

// %%%%%%%%%%%%%%%%%%%%%%%%%%%
// ALGEBRAIC FUNCTION FIELD WITH CURVE y^2 + y + x/(x^3 + x + 1) OF GENUS 2
// %%%%%%%%%%%%%%%%%%%%%%%%%%%

Kx<x> := FunctionField(F4); 
Kxy<y> := PolynomialRing(Kx);
f:=y^2 + y + x/(x^3 + x + 1);
F<c> := FunctionField(f);

//------ find places of higher degree
LP1:=Places(F,1);    // 10 degree 1 places     
LP2:=Places(F,2);   // 4 degree 2 places 

// %%%%%%%%%%%%%%%%%%%%%%%%%%%
// CHOOSE GOOD PLACE Q AND GOOD DIVISORS D1, D2
// %%%%%%%%%%%%%%%%%%%%%%%%%%%

//------ q<x>:=RandomIrreduciblePolynomial(F4,n);
q := x^5 + a*x^4 + x^3 + a^2*x^2 + a*x + 1;
Q := Decomposition(F,Zeros(Kx!q)[1])[1]; 
K<b>:=ResidueClassField(Q);
"degree of Q is ", Degree(Q); // n

//------ D1 := RandomIrreduciblePolynomial(F4,n+g-1);
D1 := x^6 + a*x^4 + a*x^2 + x + a^2;
D1:=Decomposition(F,Zeros(Kx!D1)[1])[1]; 
D1:=1*D1;

//------ D2:=RandomIrreduciblePolynomial(F4,n+g-1);
D2 := x^6 + a*x^3 + a^2*x^2 + a^2;
D2:=Decomposition(F,Zeros(Kx!D2)[1])[1]; 
D2:=1*D2;

//------ Check D1 and D2 are suitable
"D1-Q is special ? ",IsSpecial(D1-Q); // false
"dim L(D1) is ", Dimension(D1); // n
"D2-Q is special ? ", IsSpecial(D2-Q);  // false
"dim L(D2) is ", Dimension(D2); // n
"Is D1 equivalent to D2 ? ", D1 eq D2; // false
"dim L(D1+D2) is ", Dimension(D1+D2); // 11

// %%%%%%%%%%%%%%%%%%%%%%%%%%%
// CONSTRUCTION OF THE RIEMANN-ROCH SPACES
// %%%%%%%%%%%%%%%%%%%%%%%%%%%

LD1, h1 :=RiemannRochSpace(D1);
BD1 := h1(Basis(LD1)); 
LD2, h2 :=RiemannRochSpace(D2);
BD2 := h2(Basis(LD2));
LD1D2, h := RiemannRochSpace(D1+D2);

// %%%%%%%%%%%%%%%%%%%%%%%%%%%
// SET GOOD BASES
// %%%%%%%%%%%%%%%%%%%%%%%%%%%

//------ Construction of E1=Evalf(Q) and set a good basis for L(D1)
L:=[]; for i in [1..n] do L:=Append(L,ElementToSequence(Evaluate(BD1[i],Q))); end for;
E1:=Transpose(Matrix(L));
BasisLD1 := Matrix(F,1,n,BD1)*Matrix(F,E1^-1);
//BasisLD1;

//------ Construction of E2=Evalf(Q) and set a good basis for L(D2)
L:=[]; for i in [1..n] do L:=Append(L,ElementToSequence(Evaluate(BD2[i],Q))); end for;
E2:=Transpose(Matrix(L));
BasisLD2 := Matrix(F,1,n,BD2)*Matrix(F,E2^-1);
//BasisLD2;

//------ Merge the two previous bases to a basis of L(D1+D2) 
L1 := ElementToSequence(BasisLD1); 
L2 := ElementToSequence(BasisLD2);
// Concatenate L1 to L2 except the first component of L2
LL := [LD1D2!L1[i] : i in [1..n]] cat [LD1D2!L2[i] : i in [2..n]];
BasisLD1D2 := h(ExtendBasis(LL,LD1D2));

//------ In addition, we require that the two last elements of the basis are evaluated to 0
X := Transpose(Matrix(F,[ElementToSequence(Evaluate(BasisLD1D2[i],Q)): i in [2*n+g-2..2*n+g-1]]));
TT := Matrix(F,1,n,L1)*X;
BasisLD1D2[2*n+g-2] := BasisLD1D2[2*n+g-2] - ElementToSequence(TT)[1];
BasisLD1D2[2*n+g-1] := BasisLD1D2[2*n+g-1] - ElementToSequence(TT)[2];
BLD1D2 := ExtendBasis([LD1D2!BasisLD1D2[i] : i in [1..n]], LD1D2);
"Vectors of BLD1D2 are independent ? ", IsIndependent(BLD1D2);
"Basis of BLD1D2 : "; 
for i in [1..2*n+g-1] do  printf "f%o=",i; BasisLD1D2[i]; end for;
//for i in [1..2*n+g-1] do  Evaluate(BasisLD1D2[i],Q); end for;

// %%%%%%%%%%%%%%%%%%%%%%%%%%%
// CONSTRUCTION OF T AND T^-1 USING PLACES OF DEGREE 1 AND 2
// %%%%%%%%%%%%%%%%%%%%%%%%%%%

// The 9 first rows of T are the evaluation on 9 places of degree 1. 
// The two last rows are the evaluation on one place of degree 2,
// since it has 2 coordinates over F4

ST:=[]; 
for j:=1 to 9 do 
    for i:=1 to 2*n+g-1 do 
        ST:=Append(ST, Evaluate(BasisLD1D2[i], LP1[j]));
    end for; 
end for; 
STemp:= ST;

// Choose a degree 2 place such that the 11x11-matrix T has rank 11
numPlace:=0;
for j:=1 to #LP2 do
    ST := STemp;
    ST1:=[];
    for i:=1 to 2*n+g-1 do 
        eva:= ElementToSequence(Evaluate(BasisLD1D2[i], LP2[j]),F4);  
        ST1:=Append(ST1, eva[1]);
        ST1:=Append(ST1, eva[2]);
    end for; 
    for i:=1 to 4*n+2*g-2 by 2 do ST:=Append(ST, ST1[i]); end for;
    for i:=2 to 4*n+2*g-2 by 2 do ST:=Append(ST, ST1[i]); end for;
    T := Matrix(2*n+g-1,2*n+g-1, ST);	
    if Rank(T) eq 11 then numPlace:=j; break; end if;
end for;
"num of the chosen 2 degree place : ", numPlace;
"Rank of T : ", Rank(T);
TI := T^-1;

// %%%%%%%%%%%%%%%%%%%%%%%%%%%
// ALGORITHM FOR THE MULTIPLICATION
// %%%%%%%%%%%%%%%%%%%%%%%%%%%

// ============= FUNCTION MULT =============
// @parameter : VarX, VarY are the coordinates  in a canonical basis
//			of the elements of (F4)^n to multipliate
// @return : the result of VarX * VarY in the canonical basis of (F4)^n

mult := function(varX, varY) 

	//------ injection of the coordinates of X into L(D1+D2) using basis BLD1D2
	fx := VerticalJoin(varX,ZeroMatrix(F4,n+g-1,1));

	//------ injection of the coordinates of Y into L(D1+D2) using basis BLD1D2
	Y1 := [varY[1,1]] cat [0 : i in [2..n]] cat [varY[i,1] : i in [2..n]] cat [0,0];
	fy := Matrix(F4,2*n+g-1,1,Y1);

	//------ Hadamard product  u = T(fx)*T(fy)
	u:=ZeroMatrix(F4,2*n+g-1,1);
	TFX:=T*fx; 
	TFY:=T*fy;

	// the products are done in F4 for the 9 first coordinates
	for i:=1 to 9 do u[i,1]:=TFX[i][1]*TFY[i][1]; end for;
	
	// and over F16 for the other coordinates taken 2 by 2
	KK<hh>:=Parent(Evaluate(BasisLD1D2[2], LP2[numPlace])); 
	mb:=KK![TFX[10][1], TFX[11][1]]* KK![TFY[10][1], TFY[11][1]];  
	mp:=ElementToSequence(mb,F4);
	u[10,1]:=mp[1];
	u[11,1]:=mp[2];
	
	//------ E_Q(TI(u)) : T^-1 then evaluation in Q
	uu:=Matrix(F,TI*u);
	result := Evaluate(Matrix(1,2*n+g-1,BasisLD1D2)*uu,Q);
	return result;
end function; 
// =============  END FUNCTION MULT =============

// %%%%%%%%%%%%%%%%%%%%%%%%%%%
// EXAMPLES
// %%%%%%%%%%%%%%%%%%%%%%%%%%%

// Example 1 : (a+b)*(1+a*b+a*b^2) = a*b^3 + b^2 + a*b + a
X1 := Matrix(F4,n,1,ElementToSequence(a+b)); 
Y1 := Matrix(F4,n,1,ElementToSequence(1+a*b+a*b^2));
mult(X1,Y1);
 
// Example 2 : b^5 * (1 + a^2*b^3 + b^4) = b^4 + a^2*b^2 + a
X2 := Matrix(F4,n,1,ElementToSequence(b^5)); 
Y2 := Matrix(F4,n,1,ElementToSequence(1 + a^2*b^3 + b^4)); 
mult(X2,Y2);

// Example 3 : (a*b + b^2 + a*b^4) * (a*b + b^2 + a*b^4) = a^2*b^3 + a^2*b^2 + a^2*b + 1
X3 := Matrix(F4,n,1,ElementToSequence(a*b + b^2 + a*b^4)); 
Y3 := Matrix(F4,n,1,ElementToSequence(a*b + b^2 + a*b^4)); 
mult(X3,Y3);

\end{verbatim}
}

\subsection{$\F_{2^{5}}$ over $\F_{2}$}
{\small
\begin{verbatim}
// Same curve but defined over GF2. 
// We use 3 degree 1 places, 2 degree 2 places and 1 degree 4 place

n:=5; g:=2; q:=2;
F2:=GF(2);

// %%%%%%%%%%%%%%%%%%%%%%%%%%%
// ALGEBRAIC FUNCTION FIELD WITH CURVE y^2 + y + x/(x^3 + x + 1) OF GENUS 2
// %%%%%%%%%%%%%%%%%%%%%%%%%%%

Kx<x> := FunctionField(F2); 
Kxy<y> := PolynomialRing(Kx);
f:=y^2 + y + x/(x^3+x+1);
F<c> := FunctionField(f);

//------ find places of higher degree
LP1:=Places(F,1);    // 3 degree 1 places     
LP2:=Places(F,2);   // 1 degree 2 places 
LP4:=Places(F,4);   // 7 degree 4 places
#LP1;#LP2;#LP4;

// %%%%%%%%%%%%%%%%%%%%%%%%%%%
// CHOOSE GOOD PLACE Q AND GOOD DIVISORS D1, D2
// %%%%%%%%%%%%%%%%%%%%%%%%%%%

//------ q<x>:= RandomIrreduciblePolynomial(F2,n);
q := x^5 + x^3 + 1;
Q := Decomposition(F,Zeros(Kx!q)[1])[1]; 
K<b>:=ResidueClassField(Q);

//------ D1 := RandomIrreduciblePolynomial(F2,n+g-1); D1;
D1 := x^6 + x^5 + x^4 + x + 1;
D1:=Decomposition(F,Zeros(Kx!D1)[1])[1]; 
D1:=1*D1;

//------ D2:=RandomIrreduciblePolynomial(F2,n+g-1);
D2 := x^6 + x^5 + x^2 + x + 1;
D2:=Decomposition(F,Zeros(Kx!D2)[1])[1]; 
D2:=1*D2;

//------ Check D1 and D2 are suitable
"D1-Q is special ? ",IsSpecial(D1-Q); // false
"dim L(D1) is ", Dimension(D1); // n
"D2-Q is special ? ", IsSpecial(D2-Q);  // false
"dim L(D2) is ", Dimension(D2); // n
"Is D1 equivalent to D2 ? ", D1 eq D2; // false
"dim L(D1+D2) is ", Dimension(D1+D2); // 11

// %%%%%%%%%%%%%%%%%%%%%%%%%%%
// CONSTRUCTION OF THE RIEMANN-ROCH SPACES
// %%%%%%%%%%%%%%%%%%%%%%%%%%%

LD1, h1 :=RiemannRochSpace(D1);
BD1 := h1(Basis(LD1)); 
LD2, h2 :=RiemannRochSpace(D2);
BD2 := h2(Basis(LD2));
LD1D2, h := RiemannRochSpace(D1+D2);

// %%%%%%%%%%%%%%%%%%%%%%%%%%%
// SET GOOD BASES
// %%%%%%%%%%%%%%%%%%%%%%%%%%%

//------ Construction of E1 : E1=Evalf(Q) and set a good basis for L(D1)
L:=[]; for i in [1..n] do L:=Append(L,ElementToSequence(Evaluate(BD1[i],Q))); end for;
E1:=Transpose(Matrix(L));
BasisLD1 := Matrix(F,1,n,BD1)*Matrix(F,E1^-1);
//BasisLD1;

//------ Construction of E2 : E2=Evalf(Q) and set a good basis for L(D2)
L:=[]; for i in [1..n] do L:=Append(L,ElementToSequence(Evaluate(BD2[i],Q))); end for;
E2:=Transpose(Matrix(L));
BasisLD2 := Matrix(F,1,n,BD2)*Matrix(F,E2^-1);
//BasisLD2;

//------ Merge the two previous bases to a basis of L(D1+D2) 
L1 := ElementToSequence(BasisLD1); 
L2 := ElementToSequence(BasisLD2);
// Concatenate L1 to L2 except the first component of L2
LL := [LD1D2!L1[i] : i in [1..n]] cat [LD1D2!L2[i] : i in [2..n]];
BasisLD1D2 := h(ExtendBasis(LL,LD1D2));

//------ In addition, we require that the two last elements of the basis are evaluated to 0
X := Transpose(Matrix(F,[ElementToSequence(Evaluate(BasisLD1D2[i],Q)): i in [2*n+g-2..2*n+g-1]]));
TT := Matrix(F,1,n,L1)*X;
BasisLD1D2[2*n+g-2] := BasisLD1D2[2*n+g-2] - ElementToSequence(TT)[1];
BasisLD1D2[2*n+g-1] := BasisLD1D2[2*n+g-1] - ElementToSequence(TT)[2];
BLD1D2 := ExtendBasis([LD1D2!BasisLD1D2[i] : i in [1..n]], LD1D2);
"Vectors of BLD1D2 are independent ? ", IsIndependent(BLD1D2);
"Basis of BLD1D2 : "; 
for i in [1..2*n+g-1] do  printf "f%o=",i; BasisLD1D2[i]; end for;
//for i in [1..2*n+g-1] do  Evaluate(BasisLD1D2[i],Q); end for;

// %%%%%%%%%%%%%%%%%%%%%%%%%%%
// CONSTRUCTION OF T AND T^-1 USING PLACES OF DEGREE 1, 2 and 4
// %%%%%%%%%%%%%%%%%%%%%%%%%%%

// The 3 first rows of T are the evaluation on the degree 1 places.
// The 4 next rows are the evaluation on 2 places of degree;
// each evaluation is on 2 rows since it has 2 coordinates over F2
// Finally, the 4 last rows are the evaluation on 1 place of degree 4;
// this evaluation is on 4 rows since it has 4 coordinates over F2

ST:=[]; 
for j:=1 to 3 do 
    for i:=1 to 2*n+g-1 do 
        ST:=Append(ST, Evaluate(BasisLD1D2[i], LP1[j]));
    end for; 
end for; 

ST2:=[]; 
for j:=1 to 2 do  
    ST2:=[]; 
    for i:=1 to 2*n+g-1 do 
        eva:= ElementToSequence(Evaluate(BasisLD1D2[i], LP2[j]),F2);  
        ST2:=Append(ST2, eva[1]);
        ST2:=Append(ST2, eva[2]);
    end for; 
    for i:=1 to 4*n+2*g-2 by 2 do ST:=Append(ST, ST2[i]); end for;
    for i:=2 to 4*n+2*g-2 by 2 do ST:=Append(ST, ST2[i]); end for;
end for; 

ST4:=[];
for i:=1 to 2*n+g-1 do 
    eva:= ElementToSequence(Evaluate(BasisLD1D2[i], LP4[2]),F2);  
    ST4:=Append(ST4, eva[1]);
    ST4:=Append(ST4, eva[2]);
    ST4:=Append(ST4, eva[3]);
    ST4:=Append(ST4, eva[4]);
end for;
 
for i:=1 to 8*n+4*g-4 by 4 do ST:=Append(ST, ST4[i]); end for;
for i:=2 to 8*n+4*g-4 by 4 do ST:=Append(ST, ST4[i]); end for;
for i:=3 to 8*n+4*g-4 by 4 do ST:=Append(ST, ST4[i]); end for; 
for i:=4 to 8*n+4*g-4 by 4 do ST:=Append(ST, ST4[i]); end for;

T := Matrix(2*n+g-1,2*n+g-1, ST);
"Rank of T : ", Rank(T);
TI := T^-1;

// %%%%%%%%%%%%%%%%%%%%%%%%%%%
// ALGORITHM FOR THE MULTIPLICATION
// %%%%%%%%%%%%%%%%%%%%%%%%%%%

// ============= FUNCTION MULT =============
// @parameter : VarX, VarY are the coordinates  in a canonic basis
//			of the elements of (F4)^n to multiply
// @return : the result of VarX * VarY in the canonic basis of (F2)^n

mult := function(varX, varY) 

	//------ injection of the coordinates of X into L(D1+D2) using basis BLD1D2
	fx := VerticalJoin(varX,ZeroMatrix(F2,n+g-1,1));

	//------ injection of the coordinates of Y into L(D1+D2) using basis BLD1D2
	Y1 := [varY[1,1]] cat [0 : i in [2..n]] cat [varY[i,1] : i in [2..n]] cat [0,0];
	fy := Matrix(F2,2*n+g-1,1,Y1);

	//------ Hadamard product  u = T(fx)*T(fy)
	u:=ZeroMatrix(F2,2*n+g-1,1);
	TFX:=T*fx; 
	TFY:=T*fy;
	
	// the products are done in F2 for the 3 first coordinates
	for i:=1 to 3 do u[i,1]:=TFX[i][1]*TFY[i][1]; end for; // degree 1 places

	// over F4 for the next 4 coordinates taken 2 by 2
	for i:=4 to 7 by 2 do  // degree 2 places
		KK<hh>:=Parent(Evaluate(BasisLD1D2[2], LP2[(i-2) div 2])); 
		mb:=KK![TFX[i][1], TFX[i+1][1]]* KK![TFY[i][1], TFY[i+1][1]];  
		mp:=ElementToSequence(mb,F2);
		u[i,1]:=mp[1];
		u[i+1,1]:=mp[2];
	end for;

	// and over F16 for the other coordinates taken 4 by 4
	KK4<hh4>:=Parent(Evaluate(BasisLD1D2[2], LP4[ 2] ));  // degree 4 place
	mb:=KK4![TFX[8][1], TFX[9][1], TFX[10][1], TFX[11][1]]
	    * KK4![TFY[8][1], TFY[9][1], TFY[10][1], TFY[11][1]];  
	mp:=ElementToSequence(mb,F2); 
	u[8,1]:= mp[1];
	u[9,1]:= mp[2];
	u[10,1]:= mp[3];
	u[11,1]:= mp[4];

	//------ E_Q(TI(u)) : T^-1 then evaluation in Q
	uu:=Matrix(F,TI*u);
	result := Evaluate(Matrix(1,2*n+g-1,BasisLD1D2)*uu,Q);
	return result;
end function; 
// =========  END FUNCTION MULT =========

// %%%%%%%%%%%%%%%%%%%%%%%%%%%
// EXAMPLES
//%%%%%%%%%%%%%%%%%%%%%%%%%%%

// Example 1 : (1+b)*(1+b+b^2) = b^5
X1 := Matrix(F2,n,1,ElementToSequence(1+b)); 
Y1 := Matrix(F2,n,1,ElementToSequence(1+b+b^2));
mult(X1,Y1);
 
// Example 2 : b^4*(b^2+b^3) = b^20
X2 := Matrix(F2,n,1,ElementToSequence(b^4)); 
Y2 := Matrix(F2,n,1,ElementToSequence(b^2+b^3));
mult(X2,Y2);

// Example 3 : (1+b+b^3)*(1+b+b^3) = b^21
X3 := Matrix(F2,n,1,ElementToSequence(1+b+b^3)); 
Y3 := Matrix(F2,n,1,ElementToSequence(1+b+b^3));
mult(X3,Y3);

\end{verbatim}
}

\end{document}